\documentclass[11pt]{amsart} 

\usepackage[english]{babel}
\usepackage[utf8]{inputenc}
\usepackage[T1]{fontenc}
\usepackage{comment}
\usepackage{amsmath}
\usepackage{amssymb}
\usepackage{amsfonts}
\usepackage{amssymb}
\usepackage{amsthm}
\usepackage{amscd}
\usepackage{xcolor}
\usepackage{slashed}
\usepackage[hidelinks]{hyperref}
\usepackage[noabbrev,capitalise]{cleveref}
\usepackage{geometry}
\usepackage{mathtools}
\usepackage{csquotes}
\usepackage{float}
\usepackage{subcaption}
\usepackage{footnote}

\usepackage{pgfplots}
\pgfplotsset{compat=1.18}

\numberwithin{equation}{section}

\newcommand{\setsymbol}[1]{\ensuremath{\mathbb{#1}}}%
\newcommand{\R}{\setsymbol{R}}%

\newtheorem{definition}{Definition}[section]
\newtheorem{lemma}[definition]{Lemma}

\newtheorem{theorem}[definition]{Theorem}
\newtheorem{proposition}[definition]{Proposition}
\newtheorem{corollary}[definition]{Corollary}

\newtheorem{remark}[definition]{Remark}

\title{Rigidity of asymptotically hyperboloidal initial data sets with vanishing mass}

\author{Sven Hirsch}
\address{Columbia University, 2990 Broadway, New York, NY 10027, USA}
\email{sven.hirsch@columbia.edu}
\author{Hyun Chul Jang}
\address{Department of Mathematics, Texas State University, 601 University Drive, San Marcos, TX 78666, USA}
\email{hcjang@txstate.edu}
\author{Yiyue Zhang}
\address{Beijing Institute of Mathematical Sciences and Applications, Beijing, 101408, China}
\email{zhangyiyue@bimsa.cn}

\begin{document}

\begin{abstract}
In Special Relativity, the vanishing of mass indicates either the absence of energy (vacuum) or the presence of massless radiation moving at the speed of light. In General Relativity, the same principle governs the interpretation of zero total mass: asymptotically flat initial data sets (IDS) \((M^n, g, k)\) with vanishing mass correspond either to slices of Minkowski spacetime or to slices of \(pp\)-wave spacetimes that model radiation.
In contrast, we demonstrate that asymptotically hyperboloidal spin IDS with zero mass must embed isometrically into Minkowski space, with no possible IDS configurations that model radiation in this setting. 
Our result holds under general decay assumptions where the total mass is well-defined.  
The proof relies on precise decay estimates for spinors on level sets of spacetime harmonic functions and works in all dimensions.
\end{abstract}

\maketitle
\section{Introduction}

In general relativity, the total mass for an isolated gravitational model is a fundamental physical quantity and serves as important geometric invariant in differential geometry.  The \emph{vanishing mass} case is remarkably rigid in several settings, including asymptotically flat initial data sets \cite{SY1981,BeigChrusciel,CM2006,HL2020} and certain special cases of the asymptotically hyperbolic setting \cite{ACG2008,Maerten2006,HJM2020}. In this paper we prove rigidity for \emph{asymptotically hyperboloidal} initial data with zero mass.

An asymptotically hyperboloidal initial data set is a triple $(M^n,g,k)$ where $M$ is an $n$–manifold, $g$ is asymptotic to the hyperbolic metric $b$ at infinity, and $k$ is a symmetric $(0,2)$–tensor asymptotic to $g$. 
These initial data sets arise as hyperboloidal slices of asymptotically Minkowski spacetime, which are important in general relativity for studying the dynamics of gravitational waves, see for instance \cite{Penrose1982,SchoenYauBondi,BieriChrusciel}. The case when $k$ approaches zero at infinity has been extensively studied as well; such initial data sets are called asymptotically Anti-de Sitter (AdS) initial data sets, corresponding to slices of asymptotically AdS spacetime, cf.  \cite{Zhang2004,Maerten2006,CMT2006,XZ2008,WX2015}. 

For such initial data sets, the energy-momentum functional is well-defined, providing a notion of the total mass as a fundamental geometric invariant. The Positive Mass Conjecture for such initial data sets asserts that the total mass is nonnegative, provided the dominant energy condition (DEC) holds, and that the mass vanishes only for hypersurfaces of Minkowski spacetime (or correspondingly, of AdS spacetime).

In the literature, the positive mass theorem for asymptotically hyperbolic initial data sets has been extensively studied.
The \emph{umbilic} case (i.e. $k=g$) has been first studied by X. Wang \cite{Wang2001} and generalized by P. Chru\'sciel and M. Herzlich \cite{CH2003} to the case under weaker asymptotic conditions, and they proved the positive mass theorem for this case under spin assumptions including rigidity. The inequality part of the non-spin umbilic case was investigated by several authors \cite{ACG2008,CD2019,Sakovich2021,CG2021}, and the general rigidity result was proved by \cite{HJM2020} (see also \cite{CGNP2018,HJ2022}).
The positive mass inequality for general asymptotically hyperboloidal initial data sets has been established for the spin case by \cite{CJL2004,CM2006,XZ2008}. 
A non-spin proof for the three-dimensional case was given by \cite{Sakovich2021} using the Jang equation, and independently by \cite{CWY2016} using the Liu-Yau quasi-local mass. The Jang equation approach was further extended in \cite{Lundberg2023} to the dimension between $3$ and $7$.
Moreover, the positive mass inequality for general asymptotically \emph{AdS} spin initial data sets was established by \cite{Zhang2004,Maerten2006,CMT2006,XZ2008,WX2015}, and the rigidity remains open except for the case $k=0$ under some additional decay conditions. In another forthcoming work, the case of equality in this setting will be analyzed in full generality.

Despite these advances, the equality case for asymptotically hyperboloidal data—characterizing when the mass vanishes—remains unsettled except in special cases (such as $k=g$ or $E=0$). Our main theorem resolves the rigidity in the hyperboloidal setting:

\begin{theorem}
    \label{thm:rigidity}
    Let $(M^n,g,k),n\ge 3,$ be a $C^{2,a}_{-q}$ asymptotically hyperboloidal spin initial data sets of decay rate $q\in(\frac{n}{2},n]$ satisfying the dominant energy condition.\footnote{For the precise definition, we refer to \cref{SS:AH IDS}.} Suppose that the mass of $(M,g,k)$ vanishes. Then $(M^n,g,k)$ isometrically embeds into the Minkowski spacetime.
\end{theorem}

The vanishing mass case can be interpreted as follows: in relativistic physics, a massless object is physically realized only if it moves at the speed of light. Otherwise, it carries no energy and is absent. Thus, one might expect that initial data sets with vanishing mass correspond to slices of either Minkowski spacetime or non-vacuum pp-wave spacetimes, which model a plane-fronted radiation moving at the speed of light (see Figure \ref{fig2} for a schematic depiction of a pp-wave spacetime). Indeed, it is proved in the recent work of the first and third authors \cite{HZ23, HZ24} that asymptotically flat spin initial data sets with vanishing mass must be slices of either Minkowski spacetime or a non-vacuum pp-wave spacetime. In contrast, our main result shows that asymptotically hyperboloidal spin initial data sets with vanishing mass exhibit far greater rigidity, admitting only slices of Minkowski spacetime. This demonstrates that, unlike in the asymptotically flat case, no configurations modeling non-vacuum radiation are possible in the asymptotically hyperboloidal setting, highlighting the fundamental geometric differences between these two classes of initial data.

\begin{corollary}
    \label{cor:no pp-wave}
    There exist no $C^{2,a}_{-q}$ asymptotically hyperboloidal spin initial data set of decay rate $q\in(\frac{n}{2},n]$ satisfying the dominant energy condition that isometrically embeds into a non-vacuum pp-wave spacetime.
\end{corollary}

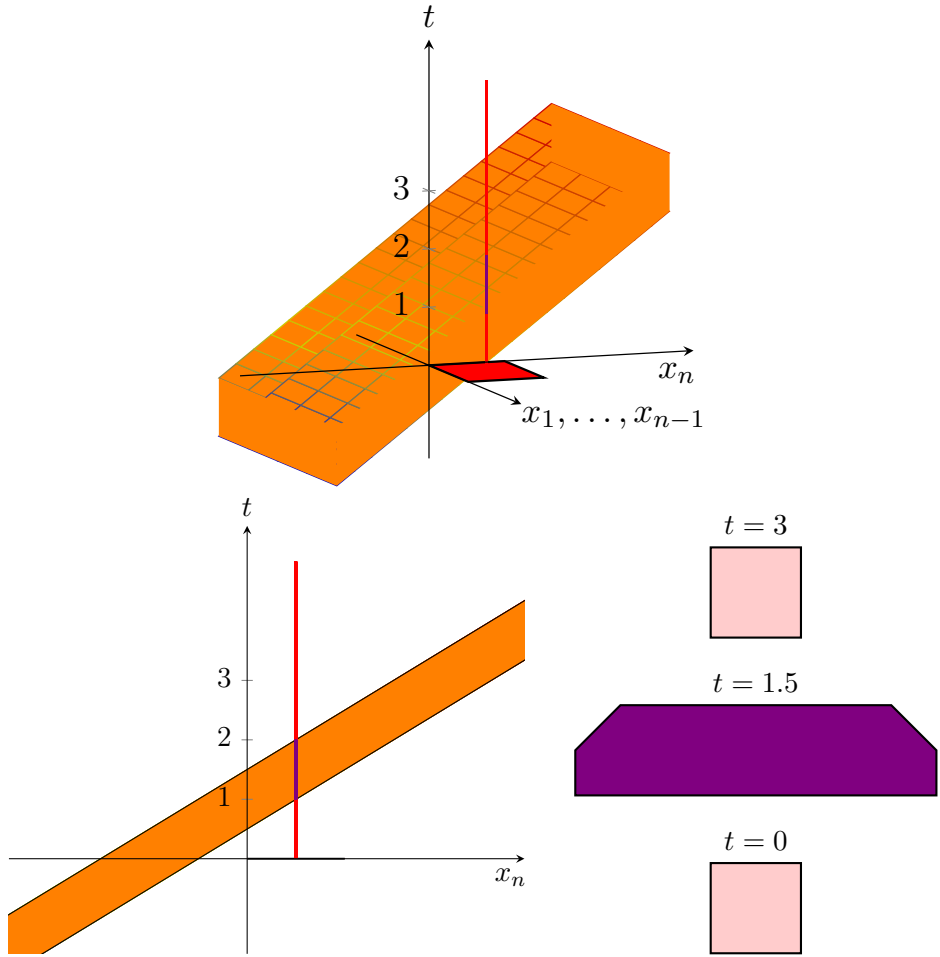
\begin{figure}
\centering
\begin{tikzpicture}[scale=1.2]
    \begin{axis}[
        axis lines = middle,
        xlabel = {\(x_1, \dots, x_{n-1}\)},
        ylabel = {\(x_n\)},
        zlabel = {\(t\)},
        xlabel style = {anchor=north east, xshift=62pt, yshift=3pt},
        ylabel style = {anchor=north west, xshift=-15pt, yshift=0pt},
        zlabel style = {anchor=south},
        domain = -2:2,
        samples = 20,
        samples y = 20,
        zmin = -1, zmax = 5,
        xmin = -1.5, xmax = 2,
        ymin = -1, ymax = 1.5,
        xtick = \empty,
        ytick = \empty,
          ztick = {0,  1,  2,3},
        zticklabels = {0, 1, 2,3},
        enlargelimits = true,
        view = {70}{10},
        axis on top
    ]

    \addplot3 [
            surf,
            color=orange,
            shader=faceted,
            opacity=0.5,
            samples=5,
            domain=-1.5:1.5,y domain=-1:1.2,
            z buffer=sort, draw=none
        ] (
            {x},
            {y},
            {2*y+1.5}
        );

    \addplot3 [
            surf,
            color=orange,
            shader=faceted,
            opacity=0.5,
            samples=5,
            domain=-1.5:1.5,y domain=-1:1.2,
            z buffer=sort, draw=none
        ] (
            {x},
            {y},
            {2*y+.5}
        );
    \addplot3 [
            fill=orange, 
            opacity=0.3, draw=none 
        ] coordinates {
            (1.5,-1,-1.5)
            (1.5,-1,-.5)
            (-1.5,-1,-.5)
            (-1.5,-1,-1.5)
        };

        \addplot3 [
            fill=orange, 
            opacity=0.3, draw=none 
        ] coordinates {
            (1.5,-1,-1.5)
            (1.5,-1,-.5)
            (1.5,1.2,3.9)
            (1.5,1.2,2.9)
        };

    \addplot3 [
            fill=orange, 
            opacity=0.3, draw=none 
        ] coordinates {
            (1.5,1.2,3.9)
            (1.5,1.2,2.9)
            (-1.5,1.2,2.9)
            (-1.5,1.2,3.9)
        };

    \addplot3 [
        fill=red,
        fill opacity=0.2,
        thick,
        line width=.25mm
    ] coordinates {
        (0, 0, 0)
        (1, 0, 0)
        (1, .5, 0)
        (0, .5, 0)
    } -- cycle;

        \addplot3 [
        thick,
        red,
        domain=0:5,
        samples=2
    ] 
    ({.5}, {.25}, {x});

    \addplot3 [
        thick,
        violet,
        domain=1:2,
        samples=2
    ] 
    ({.5}, {.25}, {x});
    
    \end{axis}
\end{tikzpicture}
\quad

\begin{tikzpicture}
    \begin{axis}[
        axis lines = middle,
        xlabel = {\(x_1, \dots, x_{n-1}\)},
        ylabel = {\(x_n\)},
        zlabel = {\(t\)},
        xlabel style = {anchor=north east, xshift=65pt, yshift=10pt},
        ylabel style = {anchor=north west, xshift=-15pt, yshift=0pt},
        zlabel style = {anchor=south},
        domain = -2:2,
        samples = 20,
        samples y = 20,
        zmin = -1, zmax = 5,
        xmin = -1.5, xmax = 2,
        ymin = -1, ymax = 1.2,
        xtick = \empty,
        ytick = \empty,
          ztick = {0,  1,  2,3},
        zticklabels = {0, 1, 2,3},
        enlargelimits = true,
        view = {90}{0},
        axis on top
    ]

    \addplot3 [
            surf,
            color=orange,
            shader=faceted,
            opacity=0.5,
            samples=5,
            domain=-1.5:1.5,y domain=-2:2,
            z buffer=sort
        ] (
            {x},
            {y},
            {2*y+1.5}
        );

    \addplot3 [
            surf,
            color=orange,
            shader=faceted,
            opacity=0.5,
            samples=5,
            domain=-1.5:1.5,y domain=-2:2,
            z buffer=sort
        ] (
            {x},
            {y},
            {2*y+.5}
        );
    \addplot3 [
            fill=orange, 
            opacity=0.3, 
        ] coordinates {
            (1.5,-2,-3.5)
            (1.5,-2,-2.5)
            (-1.5,-2,-2.5)
            (-1.5,-2,-3.5)
        } -- cycle;

        \addplot3 [
            fill=orange, 
            opacity=0.3, 
        ] coordinates {
            (1.5,-2,-3.5)
            (1.5,-2,-2.5)
            (1.5,2,5.5)
            (1.5,2,4.5)
        } -- cycle;

    \addplot3 [
            fill=orange, 
            opacity=0.3, 
        ] coordinates {
            (1.5,2,5.5)
            (1.5,2,4.5)
            (-1.5,2,4.5)
            (-1.5,2,5.5)
        } -- cycle;

    \addplot3 [
        fill=red,
        fill opacity=0.2,
        thick,
        line width=.25mm
    ] coordinates {
        (0, 0, 0)
        (1, 0, 0)
        (1, .5, 0)
        (0, .5, 0)
    } -- cycle;

        \addplot3 [
        thick,
        red,
        domain=0:5,
        samples=2
    ] 
    ({.5}, {.25}, {x});

    \addplot3 [
        thick,
        violet,
        domain=1:2,
        samples=2
    ] 
    ({.5}, {.25}, {x});
    
    \end{axis}
\end{tikzpicture}
\quad
\begin{tikzpicture}[scale=1.2]
    \fill[red!20] (1,3.5) rectangle (2,4.5);
    \draw[thick] (1,3.5) rectangle (2,4.5);
   \fill[red!20] (1,0) rectangle (2,1);
       \draw[thick] (1,0) rectangle (2,1);
 \coordinate (A) at (-.5,1.75);
    \coordinate (B) at (3.5,1.75);
    \coordinate (E) at (3.5,2.25);
    \coordinate (C) at (3,2.75);
    \coordinate (D) at (0,2.75);
    \coordinate (F) at (-.5,2.25);

    \draw[thick, fill=violet, fill opacity=0.2] (A) -- (B) -- (E) -- (C) -- (D) -- (F) -- cycle;
    \node at (1.5, 1.25) {\textbf{$t = 0$}};
\node at (1.5, 4.75) {\textbf{$t = 3$}};
\node at (1.5, 3) {\textbf{$t = 1.5$}};

\end{tikzpicture}

\caption{
On the left, an example of a pp-wave is illustrated. 
The majority of the spacetime is vacuum and coincides with Minkowski space, with the exception of the wave itself, which is highlighted as an orange beam traveling at the speed of light. This beam extends in the \(x_1, \dots, x_{n-1}\) directions with an appropriate decay profile.
Any asymptotically flat initial data set contained within this spacetime satisfies the dominant energy condition; it has zero mass but non-zero energy. 
To illustrate the properties of a pp-wave, consider the effect it has on an observer, represented by the red line. As the pp-wave passes through the observer (around time \(t=1\)), a notable elongation occurs in the \(x_n\) direction.
This stretching effect is most pronounced near the center of the wave and diminishes as \(x_1, \dots, x_{n-1}\) increase. After the wave has passed through the observer, everything returns to its original state. This phenomenon is visually represented on the right.
}
\label{fig2}
\end{figure}

Intuitively, the hyperboloidal slice in a pp-wave spacetime could either remain unaffected by the radiation, appearing as totally umbilical hyperbolic space, or be influenced in such a way that the decay conditions for asymptotically hyperboloidal initial data sets does not hold. See the beginning of Section \ref{SS: asymptotic analysis} for more details. 

The key idea of proof is the following: we first obtain several spinors satisfying the overdetermined equation $\nabla_i\psi=-\frac12k_{ij}e_je_0\psi$ by using a spinorial integral formula and the assumption that the mass of the initial data set vanishes.
This spinorial integral formula has first been shown by \cite{Wang2001} and is based on Witten's proof of the asymptotically flat positive mass theorem \cite{Witten1981}.
For completeness sake, we include the entire proof in \cref{Appendix:PMT}.
Using the family of spinors solving this overdetermined equation, we are able to construct a spacetime harmonic function $u$.
Spacetime harmonic functions have first been used in \cite{HKK}, to obtain a new proof of the spacetime positive mass theorem in Dimension 3, also see the survey paper \cite{BHKKZ}.
Moreover, they also led to many applications in purely Riemannian settings, e.g., \cite{HKKZ}.
Once having obtained this spacetime harmonic function $u$, we analyze its level sets $\Sigma_t$ and construct parallel spinor on each level set $\Sigma_t$ based on $\psi $ and $|\nabla u|$.
Due to the abundance of such spinors, we conclude that $\Sigma_t$ are flat.
While the proof closely resembles the one in the asymptotically flat case up to this point, the novel PDE and decay estimates are required to deal with the hyperbolic asymptotics. In particular, the fact that $k\to g$ at $\infty$ instead of $k\to0$ has a significant impact; for instance, it makes the non-linearity in the spacetime Laplace equation $\Delta u=-\operatorname{tr}_g(k)|\nabla u|$ much more present. It is also the main reason for \cref{cor:no pp-wave} that there are in fact no non-trivial pp-wave spacetime solutions in this asymptotically hyperboloidal case.

The paper is organized as follows:
In Section \ref{S:preliminaries}, we present several preliminaries, including mass minimizing spinors in hyperbolic space. 
In Section \ref{S:rigidity}, we obtain the improved decay estimates and prove \cref{thm:rigidity}. Since our proof directly constructs the Killing development of the given initial data set, it immediately implies \cref{cor:no pp-wave}. In Section \ref{S:example}, we discuss an example of a hyperboloidal slice in a pp-wave spacetime, illustrating why hyperboloidal hypersurfaces in a pp-wave spacetime do not satisfy the decay conditions.
Finally, we give a version of the proof of the positive mass theorem for asymptotically hyperboloidal spin initial data sets in Appendix \ref{Appendix:PMT}.

\textbf{Acknowledgements:} SH was partially supported by the National Science Foundation under Grant No. DMS-1926686, and by the IAS School of Mathematics. YZ was partially supported by NSFC grant NO.12501070 and the startup fund from BIMSA. 
The authors are also grateful to an anonymous referee whose suggestions led to various improvements.


\section{Definitions and prerequisites}\label{S:preliminaries}

\subsection{Hyperboloidal and upper half space coordinates for hyperbolic space}

For later use, we recall the two common coordinate charts for the standard hyperbolic space.
\begin{enumerate}
    \item As the upper sheet of the hyperboloid in $\mathbb{R}^{n,1}$, we have
    \begin{align*}
        \mathbb{H}^n=\{(x,t)\in\mathbb{R}^{n,1}\,|\, x_1^2+\cdots+ x_n^2-t^2=-1,t>0\}.
    \end{align*}
    We call $\mathbf{x}=(x_1,\cdots,x_n)$ the hyperboloidal coordinates. In this chart, the hyperbolic metric $b$ is written as
    \begin{align*}
        b=\frac{dr^2}{1+r^2}+r^2g_{\mathbb{S}^{n-1}}    
    \end{align*}
    where $r=|\mathbf{x}|$ and $g_{\mathbb{S}^{n-1}}$ is the standard unit sphere metric.
    
    \item Using the upper half space model, we have the coordinate chart $(y_1,\cdots, y_n), y_n>0,$ in which the hyperbolic metric $b$ is written as
        \begin{align*}
            b=\frac{1}{y_n^2}(dy_1^2+\cdots+dy_n^2).
        \end{align*}
    This coordinate chart is called the upper half space coordinates.
    
\end{enumerate}

Define the functions
\begin{align*}
    &r=|\mathbf{x}|,\quad t=\sqrt{1+r^2} \text{ from the hyperboloidal coordinates,}\\
    &\rho=\sqrt{y_1^2+\cdots+ y_{n-1}^2}\text{ from the upper half space coordinates.}
\end{align*}

The transformation rule is given as the following:
\begin{align*}
    y_i=\frac{x_i}{t+x_n}\text{ for }i=1,2,\ldots,n-1,&\quad y_n=\frac{1}{t+x_n}.
\end{align*}
Thus, we also have
\begin{align} \label{rt}
    t=\frac{1+\rho^2+y_n^2}{2y_n},\quad t+x_n=\frac{1}{y_n}.
\end{align}

\subsection{Asymptotically hyperboloidal initial data sets}\label{SS:AH IDS}

\begin{definition}[Weighted function spaces] \label{Def:weighted} 

 For $ a \in (0, 1)$, $s=0, 1, 2, \dots$, and $q\in \mathbb{R}$, the weighted H\"older space $C^{s,a}_{-q} (\mathbb{H}^n\setminus B_2)$ is  the collection of $C^{s,a}_{\mathrm{loc}}(\mathbb{H}^n\setminus B_2)$-functions $f$ that satisfy
\begin{align*}
		\| f \|_{C^{s,a}_{-q}(\mathbb{H}^n\setminus B_2)} := \sum_{|I|=0, 1,\dots, k}\sup_{x\in \mathbb{H}^n\setminus B_2} r^q |\nabla^I f(x)| +\sup_{x\in \mathbb{H}^n\setminus B_2} r^{q} [\nabla^s f]_{a; B_1(x)} <\infty,
\end{align*}
where the covariant derivatives and norms are taken with respect to the reference metric $b$, $B_1(x)\subset \mathbb{H}^n$ is a geodesic unit ball centered at $x$, and  
\[
	[ \nabla^s f]_{a; B_1(x)} := \sup_{1\le i_1, \dots, i_s\le n} \sup_{y\neq z\in B_1(x)} \frac{|e_{i_1} \cdots e_{i_s}(f) (y) - e_{i_1}\dots e_{i_s}(f) (z)|}{ d(y, z)^a}
\] 
with respect to a fixed local orthonormal frame $\{ e_1,\dots, e_n\}$ on $(\mathbb{H}^n\setminus \{0\}, b)$.
\end{definition}

\begin{definition}[Asymptotically hyperboloidal initial data sets] \label{Def:AH}
Let $(M^n,g)$ be a connected complete Riemannian manifold without boundary, and let $k$ be a symmetric $2$-tensor on $M^n$. 
We say $(M^n, g, k)$ is a $C^{s,a}_{-q}$-asymptotically hyperboloidal initial data set of decay rate $q\in (\tfrac {n}2,n]$ if it satisfies the following conditions:  
\begin{enumerate}
    \item There is a compact set $\mathcal{C}\subset M$ such that we can write $M\setminus \mathcal{C}=\cup_{\ell=1}^{\ell_0}M_{end}^{\ell}$ where the ends $M_{end}^\ell$ are pairwise disjoint and admit diffeomorphisms $\phi^\ell$ to the complement of a ball $\mathbb{H}^n \setminus B$.
    \item On each end, let
    \begin{equation*}
        e^\ell:=\phi_\ast^\ell g- b \quad\text{and}\quad \eta^\ell:=\phi_\ast^\ell (k-g).
    \end{equation*}
    Then we have for each $\ell$
    \begin{equation*}
        (e^\ell,\eta^\ell)\in C^{s,a}_{-q}\times C^{s-1,a}_{-q}.
    \end{equation*}
    \item Recall the energy density $\mu$ and the momentum density $J$:
    \begin{equation*}
        \begin{split}
            \mu:=&\frac{1}{2}\left( R+|\operatorname{tr}_g(k)|^2-|k|_g^2\right),\\
            J:=&\operatorname{div}[k-(\operatorname{tr}_gk)g].
        \end{split}
    \end{equation*}
    We require that $\mu$ and $J$ satisfy
    \begin{equation*}
            |\mu|\bar{r}\in L^1(M^n),\quad |J|_g\,\bar{r}\in L^1(M^n)
    \end{equation*}
    where $\bar{r}$ is a smooth positive function on $M$ satisfying $\bar{r}=r\circ\phi^\ell$ on each end.
\end{enumerate}

\end{definition}

We sometimes say that $(M^n,g,k)$ is $C^{s,a}_{-q}$-asymptotically hyperboloidal.
The definition of weighted function spaces from Definition \ref{Def:weighted} extends to asymptotically hyperboloidal initial data sets with the help of the diffeomorphisms $\phi^\ell$. We often omit the pullback notation $\phi_*^\ell$ when it is clear from the context.

\begin{definition}[Dominant energy condition]
  We say $(M,g,k)$ satisfies the \emph{dominant energy condition} if $\mu\ge |J|$.
\end{definition}

\begin{definition}[The energy-momentum vector]
\label{def:energy momentum vector}
    Let $M_{ext}$ be an end of an asymptotically hyperbololidal initial data set $(M^n,g,k)$.
    
    Let
    \begin{align*}
        \mathcal{N}&=\{V\in C^\infty(\mathbb{H}^n)\,|\,\mathrm{Hess}_b V=V\,b\}=\mathrm{span}_{\R}\{t,x_1,\ldots,x_n\}.
    \end{align*}
    The energy-momentum functional $\mathcal{H}_{hyp}:\mathcal{N}\to\mathbb{R}$ for the end $M_{ext}$ is defined by
    \begin{equation*}
        \mathcal{H}_{hyp}(V)=\lim_{r\to\infty}\int_{S_r(b)} \left[V(\mathrm{div}_b(e)-d\mathrm{tr}_b(e))+\mathrm{tr}_b(e+2\eta)dV-(e+2\eta)(\mathring{\nabla} V,\cdot)\, \right](\nu_b)d\mu_b,
    \end{equation*}
    where $\mathring{\nabla}$ is the connection with respect to $b$, $e=g-b$, $\eta= k-g$, and $\nu_b$ is the unit normal vector on $S_r(b)$.
    
    The energy-momentum vector $(E,\vec{P})$ is defined as
    \begin{align*}
        E:=\mathcal{H}_{hyp}(t),\quad P_i:=\mathcal{H}_{hyp}(x_i)\quad\mathrm{for }\quad i=1,\ldots,n.
    \end{align*}
    The total mass is the Minkowskian length of the mass vector, i.e., 
    \begin{align*}
        m=\sqrt{E^2-|\vec{P}|^2}
    \end{align*}
\end{definition}

\subsection{Spacetime spinor bundle}

Let $(M^n,g,k)$ be an asymptotically hyperboloidal initial data set which is spin.
We denote with $\mathcal S$ the spinor bundle of $M^n$, with $\nabla$ the induced connection, and with $\slashed D=e_i\nabla_i$ the Dirac operator. In addition, we denote by $\mathring{\nabla}$ the connection with respect to $b$ on the spinor bundle of $\mathbb{H}^n$.
\begin{definition}
   We say that $\overline{\mathcal S}=\mathcal S\oplus \mathcal S$ is the spacetime spinor bundle. Given $\phi=(\phi_1,\phi_2)\in \overline{\mathcal S}$, the space $\R^{n,1}=\operatorname{span}\{e_0,e_1,\dots,e_n\}$ acts on $\overline{\mathcal S}$ via Clifford multiplication 
   \begin{align*}
        e_l(\psi_1,\psi_2)=(e_l\psi_2,e_l\psi_1),\quad e_0(\psi_1,\psi_2)=(\mathbf{i}\psi_2,-\mathbf{i}\psi_1).
    \end{align*}
    The corresponding connection and Dirac operator are still be denoted with $\nabla, \slashed D$. The spacetime connection $\widetilde{\nabla}$ on $\overline{\mathcal S}$ is defined by
    \begin{align*}
        \widetilde{\nabla}_i\psi:=\nabla_i\psi+\frac{1}{2}k_{ij}e_je_0\psi,
    \end{align*}
    and we denote by $\widetilde{\slashed{D}}:=e_i\widetilde{\nabla}_i$ the corresponding Dirac operator.
    We call any $\widetilde{\nabla}$-parallel spinor \emph{an spacetime Killing spinor}, and the following equation is called the spacetime Killing equation:
    \begin{align*}
        \nabla_i\psi+\frac{1}{2}k_{ij}e_je_0\psi=0.
    \end{align*}
    In particular, the Clifford multiplication defined above decomposes the spacetime Killing equation on $(\mathbb{H}^n,b,b)$ into two imaginary Killing equations:
        \begin{equation*}
            \mathring{\nabla}_{i} \psi_1-\frac{\mathbf{i}}{2}e_i\psi_1=0, \quad \mathring{\nabla}_{i} \psi_2+\frac{\mathbf{i}}{2}e_i\psi_2=0.
        \end{equation*}
\end{definition}

We remark that the Clifford multiplication defined above is equivalent to the standard one used in \cite{Witten1981} up to change of basis. The second component of such a two-component spinor $\phi=(\phi_1,\phi_2)$ describes \emph{antimatter} as first observed by Oppenheimer \cite{Oppenheimer} based on Dirac's foundational 1928 paper \cite{Dirac}.

\begin{remark}
Instead of using the spacetime spinor bundle, it is also possible to work with the modified connection $\hat{\nabla}$ on $\mathcal{S}$ defined by
    \begin{align}\label{eq modified}
        \hat{\nabla}_i\phi=\nabla_i\phi+\frac{\mathbf{i}}{2}k_{ij}e_j\phi.
    \end{align}
    Then an analogous Lichnerowicz formula holds for the corresponding Dirac operator $\hat{\slashed{D}}:=e_i\hat{\nabla}_i$. 
    In fact, both approaches are equivalent.
    More precisely, a spinor $\phi\in \mathcal S(M) $ solving \eqref{eq modified} gives rise to a spinor $\psi\in\overline{\mathcal S}(M)$ solving $\tilde\nabla \psi=0$ and vice versa.
\end{remark}

 \subsection{Killing spinors of hyperbolic space}

Consider the upper half space model of hyperbolic space $\mathbb{H}^n$ 
\begin{equation*}
    b=\frac{dy_1^2+\cdots +dy_n^2}{y_n^2}
\end{equation*}
and let $e_i=y_n^{-1}\partial y_i$, $i=1,2,\dots,n$.

Recall that the standard hyperbolic space $\mathbb{H}^n$ admits imaginary Killing spinors, which are defined as the spinors satisfying the equation
\begin{align*}
    \mathring{\nabla}_i\psi=\pm\frac{\mathbf{i}}{2}e_i\psi.
\end{align*}
As introduced in \cite[Section 3]{Baum1989}, there are two types of Killing spinors on $\mathbb{H}^n$: for any Killing spinor $\psi$ on $\mathbb{H}^n$, let $f_\psi=|\psi|^2_b$ and $q_\psi=f_\psi^2-|\mathring{\nabla}f_\psi|_b^2$. 
Clearly, $q_\psi\ge 0$ and it follows directly from the Killing equation that the function $q_\psi$ is constant.
We call $\psi$ a type I Killing spinor if $q_\psi=0$ and a type II Killing spinor if $q_\psi>0$. 

In particular, for a type I Killing spinor $\psi$, we can select an upper half model such that $|\psi|_b^2=y_n^{-1}$. 
\begin{proposition}\cite[Theorem 3\&4]{Baum1989}
    A type I Killing spinor $\psi$ on $\mathbb{H}^n$ satisfies the following:
\begin{enumerate}
   
    \item  Suppose $n$ is odd.
    \begin{itemize}
        \item The restriction of $\psi$ on the horosphere $\Sigma:=\{y_n=\textit{constant}\}$, $\phi\in \mathcal{S}_\Sigma$, is parallel. 
        \item $\psi(\mathbf{y},y_n)=y_n^{-\frac{1}{2}}\phi(y)$, where $\mathbf{y}=(y_1,\cdots, y_{n-1})$.
    \end{itemize}
     \item Suppose $n$ is even.
     \begin{itemize}
         \item $\psi|_\Sigma=\psi_1\oplus (-1)^{\frac{n}{2}-1}\hat{\psi}_1$, where $\psi_1$ is a parallel spinor on $\Sigma$ and $\hat{\psi}_1=e_n\psi_1$, $e_n=y_n\partial y_n$. 
         \item $\psi(\mathbf{y},y_n)=y_n^{-\frac{1}{2}}\left[\psi_1(y)\oplus (-1)^{\frac{n}{2}-1}\hat{\psi}_1(y)\right]$.
     \end{itemize}
\end{enumerate}
\end{proposition}

According to \cite[Equation (v)]{Baum1989}, we have $e_n\psi=\mathbf{i}\psi$ when $\psi$ satisfies $\mathring{\nabla}_i \psi=\frac{\mathbf{i}}{2}e_i\psi$, and $e_n\psi=-\mathbf{i}\psi$ when $\mathring{\nabla}_i \psi=-\frac{\mathbf{i}}{2}e_i\psi$. 

It turns out that for our purposes it suffices to consider type I spinors which we will elaborate on in the next section.


\subsection{The mass minimizing spinor of the hyperboloidal PMT}

Here, we observe that the total mass can be realized by a spacetime Killing spinor consisting of type I Killing spinors.

\textbf{Notation. }In what follows, we will not distinguish between a spacetime Killing spinor $\psi^\infty$ on $\mathbb{H}^n$ and the spinor $\chi\mathcal{A}\psi^\infty$ defined on $M$ with support in the chosen end $M_{ext}$ when the context makes it clear. Here, $\mathcal{A}$ is the isomorphism introduced in \cref{SS:isomorphism} between the spacetime spinor frame bundles of $b$ and $g$, and $\chi$ is a smooth cut-off function on $M$ that equals $0$ outside $M_{ext}$ and $1$ for large $r$ in $M_{ext}$. 

\begin{proposition}\label{prop:mass minimizing spinor}
    Let $(M^n,g,k)$ be an $C^{2,a}_{-q}$-asymptotically hyperboloidal initial data set satisfying the dominant energy condition. Then there exists a spacetime Killing spinor $\psi^\infty=(\psi_1,\psi_2)$ on $\mathbb{H}^n$ such that 
    \begin{enumerate}
        \item $\psi_1$ and $\psi_2$ are type I Killing spinors on $\mathbb{H}^n$, and
        \item $\mathcal{H}_{hyp}(N_\infty)=E-|\vec{P}|$ for $N_\infty=|\psi^\infty|^2_b$.
    \end{enumerate}
    In particular, Item (1) implies that $N_\infty=|X_\infty|$ where $X_\infty$ is the $1$-form defined by $(X_\infty)_i:=\langle e_ie_0\psi^\infty,\psi^\infty\rangle$.
\end{proposition}
\begin{proof}
    Choose a function $N_\infty=t+\mathbf{a}\cdot \mathbf{x}$ where $\mathbf{a}=-\vec{P}/|\vec{P}|$. In fact, by changing coordinates if necessary, we may assume without loss of generality that $N_\infty=t+x_n=y_n^{-1}$. As discussed in the previous section, there are type I Killing spinors $\psi_1$ and $\psi_2$ with $|\psi_i|^2_b=\frac 12 N_\infty\text{ for }i=1,2,$ satisfying 
    \begin{align*}
        \mathring{\nabla}_j \psi_1-\frac{\mathbf{i}}{2}e_j \psi_1=0,\quad \mathring{\nabla}_j \psi_2+\frac{\mathbf{i}}{2}e_j \psi_2=0.
    \end{align*}
    Thus, $\psi^\infty=(\psi_1,\psi_2)$ is a spacetime Killing spinor such that $|\psi^\infty|_b^2=N_\infty$. Since $N_\infty^2=y_n^{-2}=|\mathring{\nabla}y_n^{-1}|^2_b=|X_\infty|^2_b$, it follows that $N_\infty=|X_\infty|$.
\end{proof}

We remark that it can be shown that if a Killing spinor $\psi^\infty$ minimizes the boundary integral $\mathcal H_{hyp}$ and if $|P|\ne0$, we obtain that $\psi^\infty$ must be of type I.

Below, we also collect some facts from the spinor proof of the positive mass theorem, see \cref{Appendix:PMT}.

\begin{proposition} \label{massformula}
Let $(M^n,g,k)$ be an $C^{2,a}_{-q}$-asymptotically hyperboloidal initial data set satisfying the dominant energy condition.
\begin{enumerate} 
\item Let $\psi^\infty$ be a spacetime Killing spinor in the designated end, i.e., $\mathring{\nabla}_i\psi^\infty+\frac{1}{2}e_ie_0\psi^\infty=0$, and $\psi^\infty=0$ at other ends. Then there exists a spinor $\psi\in \overline{\mathcal{S}}$ such that $\widetilde{\slashed{D}}\psi=0$ and $\psi-\psi^\infty\in H_0^1(\overline{\mathcal{S}})$.  
    \item Let $\mathcal{H}_{hyp}$ be the energy-momentum functional for $M_{ext}$ and let $\psi$ be the spinor from (1). Then we have 
    \begin{equation*}
    \mathcal{H}_{hyp}(V)
            =4\int_M \left(|\widetilde{\nabla}\psi|^2+\frac{1}{2}\langle \psi,(\mu+Je_0)\psi \rangle\right),
\end{equation*}
where $V:=|\psi^\infty|^2_b$ satisfies the equation $\mathrm{Hess}_b V=Vb$.
\end{enumerate}
\end{proposition}


\section{Rigidity}\label{S:rigidity}
In this section, we study the case of equality, i.e., the mass equals zero in the chosen end $M_{ext}$. As in the proof of \cref{prop:mass minimizing spinor}, we may assume that $\mathcal{H}_{hyp}(|v|)=0$ for $v=-t-x_n$. By \cref{prop:mass minimizing spinor} and \cref{massformula}, there exist $\psi^\infty$ and $\psi$ such that 
\begin{itemize}
    \item $\psi^\infty$ is a spacetime Killing spinor satisfying $|\psi^\infty|^2_b=|v|$, and
    \item $\widetilde{\nabla}\psi=0$ and $\psi-\psi^\infty\in H^1(\overline{\mathcal{S}})$.
\end{itemize}
It follows that $\langle e_i e_0\psi^\infty,\psi^\infty\rangle_b=\mathring{\nabla}_i v$. As in \cite{HZ24}, we define the function $N$ and the vector field $X$ by
\begin{equation*}
    N=|\psi|^2,\quad \text{and} \quad X_i=\langle e_ie_0\psi,\psi\rangle.
\end{equation*}

\begin{lemma}\label{|X|=N}
We have
\begin{enumerate}
    \item  $|X|=N\ne0$. 
    \item $\nu \psi_1=\mathbf{i}\psi_1$ and $\nu \psi_2=-\mathbf{i}\psi_2$ where $\nu=|X|^{-1}X$.
\end{enumerate}
\end{lemma}

\begin{proof}   
(1) We have 
\begin{equation} \label{XN}
    \nabla_i X_j=-k_{ij} N,\quad \text{and} \quad \nabla_i N=-k_{ij}X_j, 
\end{equation}
which implies $\nabla(N^2-|X|^2)=0$.

Since $|\psi-\psi^\infty|\in L^2(M^n)$ and $|\psi^\infty|^2=t+x_n$, there exists a sequence of points $\{p_i\}\in M^n$ such that $p_i\to \infty$, $|\psi^\infty(p_i)|\to 0$, and
$|\psi(p_i)|\to 0$. Hence, we have $N=|X|$ since $N_\infty=|X_\infty|$.  

Let $C$ be a constant such that $|k|\le C$ on $M^n$. 
Then $|\nabla_i N|\le C|X|=CN$. By a standard ODE argument, $N$ is nowhere vanishing.

(2)  This follows since $N=|X|$, $X_i=\langle \mathbf{i}e_i \psi_2,\psi_2\rangle-\langle \mathbf{i}e_i\psi_1,\psi_1\rangle$ and $N=|\psi_1|^2+|\psi_2|^2$.
\end{proof}

\subsection{Decay estimates at infinity}

Now we obtain the pointwise decay estimates using the gradient equation of the spinor.

\begin{lemma} \label{decay est}
    Suppose the decay rate of IDS is $q>\frac{n}{2}$. Then we have $\psi-\psi^\infty\in O_{2,a}(r^{-q}|v|^\frac{1}{2})$.
\end{lemma}

\begin{proof}
    Suppose $\widetilde{\nabla}_j\phi:=\nabla_j \phi+\frac{1}{2}k_{jl}e_l e_0 \phi$. Let $\zeta:=\psi-\psi^\infty$. Since $\widetilde{\nabla}_j \psi^\infty=O(|v|^\frac{1}{2}r^{-q})$, we have $\widetilde{\nabla}_j \zeta=-\widetilde{\nabla}_j \psi^\infty =O(|v|^\frac{1}{2}r^{-q})$. Thus, assuming the matrix norm of $k$ less than $2$ in the asymptotic region, and choosing a unit vector $e$ such that $|\nabla |\zeta||=|\nabla_e|\zeta||$, we have 
    \[\frac{1}{2}|\nabla|\zeta|^2|=|\langle\nabla_e\zeta,\zeta\rangle|\le |\langle\tilde{\nabla}_e\zeta,\zeta\rangle|+\frac{1}{2}|\langle k(e,e_l) e_le_0\zeta,\zeta \rangle| \le |\zeta|^2+ C|v|^\frac{1}{2}r^{-q}|\zeta|.\] 
Note that $|\nabla|\zeta||=0$ when $|\zeta|=0$. Therefore, $|\nabla|\zeta||\le |\zeta|+C|v|^\frac{1}{2}r^{-q}$.

Let $(r,\theta_1,\theta_2,\cdots \theta_{n-1})$ be the hyperboloidal coordinate system on $M_{ext}$, where $(\theta_1,\theta_2,\cdots, \theta_{n-1})$ is a coordinate system on $S^{n-1}$ such that $v=-t-r\cos \theta_1$.

    For any $p\in M_{ext}$, 
  let $\gamma(s): [0,\infty)\to M^n$ be a curve parametrized by its length such that $\gamma(0)=p$, $\gamma'(s)=\frac{\partial_r}{|\partial_r|}|_{\gamma(s)}$ and $\lim_{s\to \infty}r(\gamma(s))=\infty$. Then we have 
    \begin{equation*}
        \frac{d}{ds} |\zeta|(\gamma(s))\ge -|\zeta|(\gamma(s))-C|v|^\frac{1}{2} r^{-q}.
    \end{equation*}
    Since  
    $r(\gamma(s))=O(e^{s})$, we have 
    \begin{equation*}
        \frac{d}{ds} \left[e^{s}|\zeta|(\gamma(s))\right]\ge -C_1e^{(-q+\frac{3}{2})s}\frac{|v|^\frac{1}{2}}{r^\frac{1}{2}}.
    \end{equation*}
    Observe that $\frac{|v|}{r}=\sqrt{1+r^{-2}}+\cos \theta_1$ is decreasing in $r$. Therefore, $\frac{|v|}{r}$ is decreasing along $\gamma(s)$ as $s$ increases.
    Then
    applying $\zeta\to 0$ as $r\to \infty$,  
    \begin{equation*}
    \begin{split}
        |\zeta|(\gamma(s))\le& e^{-s}\int_s^\infty C_1 e^{(-q+\frac{3}{2})\tau}\frac{|v|^\frac{1}{2}}{r^\frac{1}{2}}(\gamma(\tau))d\tau 
        \\ \le & \frac{|v|^\frac{1}{2}}{r^\frac{1}{2}}(\gamma(s))\cdot e^{-s}\int_s^\infty C_1 e^{(-q+\frac{3}{2})\tau} d\tau
        \\=&O(r^{-q}|v|^\frac{1}{2})
    \end{split}
    \end{equation*}
    Therefore, $|\zeta|=O(r^{-q}|v|^\frac{1}{2})$. We conclude the desired estimates by applying the standard elliptic estimates.
\end{proof}

Next, we establish improved decay estimates on a slab $\{ -\mathcal{N}\le v\le -\frac{1}{\mathcal{N}}\}$, where  $\mathcal{N}$ is sufficiently large.  Within this slab, $r$ has the same growth rate as $\rho^2$. For the sake of clarity, we use $\mathcal{O}$ to represent the decay rate on the level set of $v$ or $u$. In $\mathbb{H}^n$, the level sets of $v$ are horospheres, therefore, level sets of $v$ in $M^n$ are asymptotically flat. Moreover, the level sets $\Sigma$ of $u$ are also asymptotically flat, as established by the decay estimates in Corollary \ref{u uni}.  For $f\in \mathcal{O}_i(\rho^{-q})$, it means $(\nabla^{\Sigma})^j f=\mathcal{O}(\rho^{-q-j})$ for $j\le i$, which differs from the notation $O_i(r^{-q})$ in hyperbolic space.

\begin{lemma}
\label{u-v}
\begin{enumerate}
\item There exists a function $u$ defined in $M_{ext}$ such that $X=\nabla u$ in $M_{ext}$  and $u$ satisfying $\nabla (u-v)=O_{2}(r^{-q}|v|)$.

    \item $u$ can be chosen to satisfy the pointwise estimate $u-v=O(|v|^\frac{3}{2}r^{\frac{1}{2}-q})$. In particular,  on the slab $\{-\mathcal{N}\le v\le -\frac{1}{\mathcal{N}}\}$, $u-v=\mathcal{O}_1(\rho^{-2q+1})$. 
    \item For $|\tau|+|\tau|^{-1}$ sufficiently large, $\{u=\tau\}$ is a graph over $\{v=\tau\}$. In the region where $\rho$ is sufficiently large,   $\{u=\tau\}$ is also  a graph over $\{v=\tau\}$. 
\end{enumerate}
\end{lemma}
\begin{proof} 
(1) Since $\nabla_{i}X_j=-k_{ij}N=\nabla_j X_i$,  $X$ is locally exact. Note that $M_{ext}$ is simply connected, therefore, we can find a function $u$ defined in this region such that $X=\nabla u$.

Furthermore, recall that  $|\psi^\infty|_b=|v|^\frac{1}{2}$, we have
\begin{equation*}
\begin{split}
     \nabla_i(u-v)=&\langle e_ie_0\psi,\psi\rangle- \langle e_ie_0\psi^\infty,\psi^\infty\rangle +O_2(r^{-q}|v|)
     \\=& \langle e_ie_0(\psi-\psi^\infty),\psi\rangle+\langle e_ie_0\psi^\infty,\psi-\psi^\infty\rangle+O_2(r^{-q}|v|)
     \\=&O_2(r^{-q}|v|).
\end{split}
\end{equation*}

(2) We  prescribe $u$ such that $u-v\to 0$ as $v\to 0$ at $\infty$. Let $p_0\in M^n$ satisfy $\rho(p_0)$ is sufficiently large.    
    We integrate $\nabla(u-v)$ from $p_0$ to the infinity where $v\to 0$, i.e., $y_n\to\infty$. Using the expression of $t$ in \eqref{rt}, we have
    \begin{equation*}
        \begin{split}
            |u(p_0)-v(p_0)|\le & \int_{y_n(p_0)}^\infty |\partial_{y_n}(u-v)| dy_n
           \\ \le& C \int_{y_n(p_0)}^\infty |\nabla (u-v)| y_n^{-1}dy_n
           \\ \le & C \int_{y_n(p_0)}^\infty t^{-q}|v| y_n^{-1}dy_n 
           \\ \le  & C \int_{y_n(p_0)}^\infty (\rho^2+y_n^{2}+1)^{-q} y_n^{q-2}dy_n
           \\ \le & C\int_{y_n(p_0)}^\infty (\rho^{q-1}+y_n^{q-1}+1)^{-\frac{2q}{q-1}}y_n^{q-2} dy_n
           \\ \le& C(\rho^{q-1}(p_0)+y_n^{q-1}(p_0)+1)^{-\frac{q+1}{q-1}}
        \end{split}
    \end{equation*}
    where $C$ is a constant that varies from line by line. Therefore, $u-v=\mathcal{O}(\rho^{-q-1})$.  Now  $u-v\to 0$ as $\rho\to \infty$ on the level set of $v$. By integrating $\nabla(u-v)$ on the level set of $v$, we obtain, for any $p\in M_{ext}$, that
    \begin{equation*}
        \begin{split}
            |u(p)-v(p)|\le& \int^\infty_{\rho(p)} |\nabla(u-v)|\cdot|\partial_\rho| d\rho
            \\  \le & C\int^\infty_{\rho(p)} t^{-q}|v|\cdot y_n^{-1} d\rho
            \\  = & C\int^\infty_{\rho(p)} (\rho^2+y_n^2+1)^{-q} y_n^{q-2} d\rho
            \\ \le &
            C\int^\infty_{\rho(p)} (\rho+y_n+1)^{-2q} y_n^{q-2} d\rho
            \\ \le & C(\rho(p)+y_n+1)^{1-2q}y_n^{q-2}
            \\ \le & C t(p)^{\frac{1}{2}-q}|v|^{\frac{3}{2}}.
        \end{split}
    \end{equation*}
    Thus, on the level set of $v$, using $t=\mathcal{O}(\rho^2)$, we obtain $u-v=\mathcal{O}(\rho^{1-2q})$. Combined with $\nabla(u-v)=O(r^{-q}|v|)$, this yields $u-v=\mathcal{O}_1(\rho^{1-2q})$.

  (3) Recall that $v=-y_n^{-1}$. Consider the coordinate $\{v, y_1,\dots, y_{n-1}\}$. We
  define a map $\Phi:\{u=\tau\}\to \{v=\tau\}$ by $p\mapsto (\tau,y_1(p),\dots, y_{n-1}(p))$ in the region where either $\rho$ or $|\tau|+|\tau|^{-1}$ is sufficiently large.  
  
  If $\Phi$ is not injective, i.e., there exist $v_1\neq v_2$ and $y_1,\dots, y_{n-1}$  such that $u(v_1,y_1,\dots y_{n-1})=u(v_2,y_1,\dots y_{n-1})$, then $\partial_v u=0$ at  $(v_3,y_1,\dots, y_{n-1})$ for some $v_3$ between $v_1$ and $v_2$. 
  However, $|\nabla (u-v)|\le C|v|r^{-q}$ and $|\partial_v|=y_n+O(y_nr^{-1})$ imply that $|\partial_v u|\ge |\partial_v v|-|\partial_v(u-v)|\ge 1-Cr^{-q}$, which leads to a contradiction. Therefore, $\Phi$ is injective.
  
  The map
    $\Phi$ is surjective by the estimate $u-v=O(|v|^\frac{3}{2}r^{\frac{1}{2}-q
    })$ and the intermediate value theorem. 
    Moreover, $u-v=O(|v|^\frac{3}{2}r^{\frac{1}{2}-q
    })$ implies that $D\Phi$ is bijective. This completes the proof.
\end{proof}

\subsection{Topology of $M^n$}
\begin{lemma} \label{one end}
    $M^n$ has only one end.
\end{lemma}
\begin{proof}
    Suppose $M^n$ has at least two ends. Since $\psi-\psi^\infty\in L^2(M^n)$ and $\psi^\infty$ is chosen to be the zero spinor at other ends, $\psi$ is $L^2$ integrable there. Note that the gradient equation implies
    $\nabla_i|\psi|\le (\frac{1}{2}|k|+\frac{1}{2})|\psi|$. Using the fact $|k|\le C$ at the asymptotic region, we obtain $|\psi|\ge Cr^{-1}$. This leads to a contradiction to the  $L^2$ integrability. (see also \cref{lemma:trivial kernel})
\end{proof}
\begin{theorem}
    $M^n$ is topologically trivial.
\end{theorem}
\begin{proof}
    
Observe that $X$ satisfies $\nabla_i X_j-\nabla_j X_i=0$, thus, $X$ is exact locally. Then according to the global Frobenius theorem \cite[Theorem 19.21]{JohnLee}, there is a foliation of $M^n$ such that $X$ is normal to each leaf.  

According to Lemma \ref{decay est}, 
 the leaves $\{u=\tau\}$ are graphs over $\{v=\tau\}$ for $|\tau|+|\tau^{-1}|$ sufficiently large.  Therefore, 
we can choose a sufficient large cube so that the top and bottom faces are leaves, while side faces are transverse to the foliation. The cube is compact because of Lemma \ref{one end}. Furthermore, the normal vector field $X$ is nowhere vanishing, which indicates that the foliation $\mathcal{F}$ is transversely orientable.
Then the trivial topology of $M^n$ follows by Reeb's stability theorem with boundary, see \cite[Theorem 3.1, Page 112]{godbillon1991feuilletages}.
\end{proof} 
Thus, $X$ is globally exact due to the trivial topology of $M^n$. In combination with Equation \eqref{XN},  Lemma \ref{|X|=N}(1) and \ref{u-v}, 
a direct corollary follows.
\begin{corollary} \label{u uni}
    There exists a globally defined function $u$ satisfying
    \begin{equation} \label{wuy}
        \nabla^2 u =-|\nabla u|k,\quad
        \nabla(u-v)=O_{2}(r^{-q}|v|)\quad \text{and} \quad u-v=O(|v|^{\frac{3}{2}}r^{\frac{1}{2}-q}). 
    \end{equation}
\end{corollary}

\subsection{Asymptotic analysis}\label{SS: asymptotic analysis}

Heuristically, there are no hyperboloidal pp-waves because of the following reason:
The proof is based upon the maximum principle applied on the level sets of spacetime harmonic functions.
Note that harmonic functions in Euclidean space $\mathbb R^{n}$ decay like $\rho^{2-n}$ and in hyperbolic space decay like $r^{1-n}$.
On the other hand, asymptotically flat manifolds have decay rates $q>\frac{n-2}2$ and asymptotically hyperboloidal manifolds have decay rates $q>\frac{n}2$.
While so far the settings are very similar, there is one key difference:
If a hyperbolic initial data set is foliated by the level sets of a spacetime harmonic function, one obtains improved decay estimates on the level sets.
More precisely, the decay rates are doubled:
To see this, consider for instance $\mathbb H^3\subset \R^{3,1}$ given by $x^2+y^2+z^2+1=t^2$.
Fixing a level set of a spacetime harmonic function, e.g. $\{t-z=1\}$, we obtain
\begin{align*}
    1+x^2+y^2=t^2-z^2=(t+z)(t-z)=t-z=2t-1.
\end{align*}
Hence $r\simeq t\simeq x^2+y^2=\rho^2$ on this level set.
Therefore, a function having a decay rate of $r^{-q}$ in the hyperboloidal initial data set has a decay rate of $\rho^{-2q}$ on this level set. 
Moreover, assuming the function is superharmonic, we conclude that this function is constant by the maximum principle.

Following the proof in \cite[Section 4.3] {HZ24}, we show the flatness of the level sets.
\begin{proposition}
    The level set $\Sigma$ of $u$ is flat. Moreover, the second fundamental form of $\Sigma$ is $-k|_\Sigma$.
\end{proposition}
\begin{proof}
Note that for $\psi=(\psi_1,\psi_2)$, $\psi_1$ and $\psi_2$ satisfies
\begin{equation*}
    \nabla_i \psi_1-\frac{\mathbf{i}}{2}k_{ij}e_j \psi_1=0 \quad \textit{and}\quad 
    \nabla_i \psi_2+\frac{\mathbf{i}}{2}k_{ij}e_j \psi_2=0. 
\end{equation*}
Moreover, Lemma \ref{|X|=N} states
$\nu \psi_1=\mathbf{i}\psi_1$ and
 $\nu \psi_2=-\mathbf{i}\psi_2$. Thus, 
 \begin{equation*}
    \nabla_i \psi_1-\frac{1}{2}k_{ij}e_j\nu \psi_1=0 \quad \textit{and}\quad 
    \nabla_i \psi_2-\frac{1}{2}k_{ij}e_j\nu \psi_2=0, 
\end{equation*}
where $\nu=|\nabla u|^{-1}\nabla u$ is the unit normal vector to the level sets of $u$. 

Then following the calculation in  \cite[Section 4.3] {HZ24}, we have $|\nabla u|^{-\frac{1}{2}}\psi_1$ is a parallel spinor on $\Sigma$ when $n$ is odd. The even case is similar and addressed in \cite{HZ24}.  

Based on the decay estimates of 
$u-v$, we conclude that the level set of 
$u$ is asymptotically flat. Therefore, the level set is flat due to the existence of a parallel spinor, and the second fundamental form follows by $\nabla^2 u=-|\nabla u|k$.
\end{proof}

Next, we construct a coordinate system based on the flatness of the level sets.

\begin{proposition} \label{Y decay}
There exists a coordinate system $(u,w_1,\dots, w_{n-1})$ such that
\begin{equation} \label{uYw}
    g=(|\nabla u|^{-2}+|Y|^2)du^2+2\sum_{\alpha=1}^{n-1}Y_\alpha du dw_\alpha+ \sum^{n-1}_{\alpha=1} w_\alpha^2,
\end{equation}
where $Y$ is a vector tangent to the level sets $\Sigma$ of $u$. Moreover, 
$Y_\alpha=-\frac{w_\alpha}{u}+\mathcal{O}_1(\rho^{2-n-c})$, for some small $c\in(0,\frac{1}{2}\min\{q-\frac{n}{2},a\})$.
\end{proposition}
\begin{proof} 
Recall that $\nabla(u-v)=O_2(r^{-q}|v|)$, $v=-\frac{1}{y_n}$ and
\begin{equation*}
    g=\frac{dy_1^2+\cdots +dy_n^2}{y_n^2}+O_{2,a}(r^{-q}).
\end{equation*}
We then have the induced metric $g_\Sigma$ on $\Sigma$ satisfying
  \begin{equation} \label{gSigma}
      g_\Sigma=\frac{dy_1^2+\cdots + dy_{n-1}^2}{y_n^2}+O_{2,a}(r^{-q}).
  \end{equation}

Recall that $\mathbf{n}=\frac{\nabla v}{|\nabla v|}$ and $\nu=\frac{\nabla u}{|\nabla u|}$, 
\begin{equation*}
    \nu-\mathbf{n}=\frac{|\nabla v|\nabla u-|\nabla u|\nabla v}{|\nabla u|\cdot |\nabla v|}=\frac{\nabla (u-v)}{|\nabla u|}+\frac{(|\nabla v|-|\nabla u|)\nabla v}{|\nabla u|\cdot |\nabla v|}=O_{1,a}(r^{-q}),
\end{equation*}
and
\begin{equation*}
    |\nabla y_\alpha|=|v|^{-1}(1+O_1(r^{-q})), \quad 
    \nabla^2 y_\alpha=-|v|(dy_n\otimes dy_\alpha+dy_\alpha\otimes dy_n)+O_{1,a}(|v|^{-1}r^{-q}),
\end{equation*}
then combined with $H^\Sigma=-\operatorname{tr}(k|_\Sigma)=1-n+O_{1,a}(r^{-q})$, we have
   \begin{equation} \label{uyalpha}
       \Delta_{\Sigma} (uy_\alpha)=
       u(\Delta-\nabla_{\nu\nu}-H^\Sigma\nabla_\nu)y_\alpha
       =O_{1,a}(r^{-q}).
   \end{equation}
In what follows, we construct a function $w_{\alpha,R}$ for each $\alpha=1,\ldots,n-1$ satisfying 
    \begin{equation*}
        \Delta_{\Sigma} w_{\alpha,R}=0,\quad  w_{\alpha,R}=-uy_\alpha \text{ on } \Sigma\cap \{\rho=R\}.
    \end{equation*}
    Let $\xi_{R}:=w_{\alpha,R}+uy_\alpha$. Then 
    \begin{equation*}
        \Delta_\Sigma \xi_{R}=O_{1,a}(\rho^{-2q}),\quad \xi_{R}=0 \text{ on } \Sigma\cap \{\rho=R\}.
    \end{equation*}
By applying $r^{-q}=\mathcal{O}(\rho^{-2q})$ on each level set, and recall that $O_{1,a}(r^{-q})$ indicates the decay rate remains unchanged under differentiation and difference quotients, we have $\Delta_\Sigma \xi_{R}=\mathcal{O}_{0,c}(\rho^{-n-c})$, for any $c\in (0,\frac{1}{2}d)$ and $d:=\min\{q-\frac{n}{2},a\}$. Therefore, $\xi_{R}=\mathcal{O}_{2,c}(\rho^{2-n-c})$. 

To obtain the regularity and decay rate of $w_\alpha$ in $\partial_u$ direction, we differentiate $\Delta_\Sigma w_\alpha=0$ in this direction. 
Let $\tilde{\Gamma}_{\alpha\beta}^\gamma$ be the Christoffel symbol with respect to the coordinate system $\{y_1,\cdots, y_{n-1}\}$ on $\Sigma$.  Thus, 
\begin{equation}\label{xiR}
\begin{split}
    \partial_u\Delta_{\Sigma} \xi_{R}=&
    \partial_u\left[g_\Sigma^{\beta\gamma}(\partial_{\beta\gamma}-\tilde{\Gamma}_{\beta\gamma}^\delta\partial_\delta)\xi_{R}\right]
    \\ =& \Delta_\Sigma \partial_u \xi_{R}+(\partial_u g_\Sigma^{\beta\gamma})(\partial_{\beta\gamma}-\tilde{\Gamma}_{\beta\gamma}^\delta\partial_\delta)\xi_{R}- g_\Sigma^{\beta\gamma}(\partial_u\tilde{\Gamma}_{\beta\gamma}^\delta)\partial_\delta \xi_{ R}. 
\end{split}
\end{equation}
 Applying \eqref{gSigma}, we have $\tilde{\Gamma}_{\alpha\beta}^\gamma=O_{1,a}(y_n^{-1}r^{-q})$, thus, $\tilde{\Gamma}_{\alpha\beta}^\gamma=\mathcal{O}_{1,c}(\rho^{1-n-c})$. Combined with the facts that $\partial_u\Delta_{\Sigma} \xi_{R}=\mathcal{O}_{0,c}(\rho^{-n-c})$, $\partial_u g_\Sigma^{\beta\gamma}=\mathcal{O}_{1,c}(\rho^{1-n-c})$ and $\xi_{R}=\mathcal{O}_{2,c}(\rho^{2-n-c})$, 
   , we have 
   $\Delta_\Sigma (\partial_u\xi_{R})= \mathcal{O}_{0,c}(\rho^{-n-c})$. Hence, $\partial_u \xi_{R}=\mathcal{O}_{2,c}(\rho^{2-n-c})$.

Moreover, using $\tilde{\Gamma}^\gamma_{\alpha\beta}=O_{1,a}(\rho^{-2q})$, $\partial_u \Delta_\Sigma \xi_R=O_{0,a}(\rho^{-2q})$ and $\partial_u g_\Sigma^{\beta\gamma}=O_{1,a}(\rho^{-2q})$, thus, Equation \eqref{xiR} implies $\Delta_\Sigma \partial_u \xi_R=O_{0,a}(\rho^{-2q})$. Hence, $\Delta_\Sigma \partial_u \xi_R=\mathcal{O}_{0,d}(\rho^{-n-d})$

Next, we take the difference quotient to obtain the estimate of Holder norm of $\partial_u \xi_{\alpha,R}$. 
Let $\zeta_R=\varepsilon^{-c}(\partial_u\xi_{R}(y,u+\varepsilon)-\partial_u\xi_{R}(y,u))$. We have  
\begin{align*}
    \Delta_\Sigma \zeta_{R}=&\varepsilon^{-c}(\Delta_{\Sigma_{u+\varepsilon}}\partial_u\xi_{R}(y,u+\varepsilon)-\Delta_{\Sigma_u}\partial_u\xi_{R}(y,u))
    -\varepsilon^{-c}(\Delta_{\Sigma_{u+\varepsilon}}-\Delta_{\Sigma_u})\partial_u\xi_{R}(y,u+\varepsilon)
    \\=& \varepsilon^{d-c}\mathcal{O}_{0,d-c}(\rho^{-n-c})+\varepsilon^{1-c}\mathcal{O}_{0,c}(\rho^{1-2n-c}),
\end{align*}
where we use  $\Delta_\Sigma \partial_u \xi_R=O_{0,d}(\rho^{-n-d})$ and
\begin{align*}
        &  \varepsilon^{-c} \left( \Delta_{\Sigma_{u+\varepsilon}} - \Delta_{\Sigma_u} \right) \partial_u\xi_{R}(y, u+\varepsilon)  \\
        = & \varepsilon^{-c} \left( g_{\Sigma_{u+\varepsilon}}^{\beta\gamma} - g_{\Sigma_u}^{\beta\gamma} \right) \partial_{\beta\gamma} \partial_u\xi_{R}(\mathbf{x}, u+\varepsilon) \\
        &\quad + \varepsilon^{-c}  \left( g^{\beta\gamma}_{\Sigma_{u+\varepsilon}} \tilde{\Gamma}^{\alpha}_{\beta\gamma}(\mathbf{x}, u+\varepsilon) - g^{\beta\gamma}_{\Sigma_u} \tilde{\Gamma}^{\alpha}_{\beta\gamma}(\mathbf{x}, u) \right) \partial_\alpha \partial_u\xi_{R}(y, u+\varepsilon)  \\
        = & \varepsilon^{1-c}\mathcal{O}_{0,c}(\rho^{1-2n-c}).
\end{align*}
Therefore,  $\Delta_\Sigma \zeta_R=\varepsilon^{d-c}\mathcal{O}_{0,c}(\rho^{-n-c})$, then $\zeta_R=\varepsilon^{d-c}\mathcal{O}_{0,c}(\rho^{2-n-c})$. 

By passing a subsequence, $\xi_{R}$ converges to $\xi$ uniformly $C^{2,c'}(\Sigma)$ for any fixed $c'\in (0,c)$, and $\partial_u \xi_{R}$ converges to $\partial_u\xi$ uniformly $C^{3,c'}(\Sigma)$. 
Thus, $\xi=\mathcal{O}_{2}(\rho^{2-n-c})$, $\partial_u \xi=\mathcal{O}_{2}(\rho^{2-n-c})$. 

 Recall that under the coordinate $\{u,y_1,\cdots y_{n-1}\}$, $g_{ij}=y_n^{-2}\delta_{ij}+O(\rho^{-2q})$, where $i,j=1,...,n-1$ or $u$. Then $\partial_u= (1+O(\rho^{-2q}))y_n^{-1}\nu+(\partial_u)^T$ and $(\partial_u)^T=O(\rho^{-2q})$, where $(\partial_u)^T$ is the tangential part of $\partial_u$ with respect to the level sets. Thus, $\nabla_\nu \xi_\alpha=\mathcal{O}_2(\rho^{2-n-c})$.

Note that Equation \eqref{wuy} implies $|\nabla u|=-u+\mathcal{O}_1(\rho^{1-2q})$ and $\langle du, dy_\alpha\rangle=\mathcal{O}_1(\rho^{1-2q})$, we get
\begin{equation*}
\begin{split}
    \nabla_\nu (-uy_\alpha)=&\langle \frac{\nabla u}{|\nabla u|},  \nabla (-uy_\alpha)\rangle
    \\=&-y_\alpha |\nabla u|-\frac{u}{|\nabla u|} \langle du,dy_\alpha \rangle
    \\=&-y_\alpha |\nabla u|+\mathcal{O}_1(\rho^{1-2q})
    \\=& -w_\alpha+\mathcal{O}_1(\rho^{2-2q}).
\end{split}
\end{equation*}
Let $w_\alpha=\xi-uy_\alpha$. Then  we have $\nabla_\nu w_\alpha=-w_\alpha+\mathcal{O}_1(\rho^{2-n-c})$. 

Next, we will show that  $\{ w_1,\cdots,w_{n-1}\}$ forms a flat coordinate system on $\Sigma$. Note that $w_1,\cdots,w_{n-1}$ are harmonic functions on $\Sigma$. Furthermore, 
\[\langle \nabla^\Sigma w_\alpha,\nabla^\Sigma w_\beta\rangle=\langle \nabla^\Sigma(uy_\alpha)+\mathcal{O}_1(\rho^{1-n-c}), \nabla^\Sigma(uy_\beta)+\mathcal{O}_1(\rho^{1-n-c})\rangle=\delta_{\alpha\beta}+\mathcal{O}_1(\rho^{1-n-c}).\]
Thus, $|\nabla^\Sigma w_\alpha|$ is bounded. Then $w_\alpha$ is a linear function on $\Sigma$ which implies that $\langle \nabla^\Sigma w_\alpha,\nabla^\Sigma w_\beta\rangle$ is constant. Hence, $\langle \nabla^\Sigma w_\alpha,\nabla^\Sigma w_\beta\rangle=\delta_{\alpha\beta}$, i.e., the metric $g$ satisfies Equation \eqref{uYw}.
   
   Finally,  taking the inverse of the metric in \eqref{uYw}, we have $ \langle du,dw_\alpha\rangle=-|\nabla u|^{2}Y_\alpha$, therefore,
   \begin{equation*}
       Y_\alpha=-|\nabla u|^{-1}\nabla_\nu w_\alpha=-\frac{w_\alpha}{u}+\mathcal{O}_1(\rho^{2-n-c}).
   \end{equation*}
\end{proof}

Similar to \cite[Lemma 6.5]{HZ24}, we have the following lemma.
\begin{lemma} \label{lemma:ell function}
On each level set,    $d^\Sigma Y_\alpha=0$. In particular, $Y_\alpha= \nabla^\Sigma \ell$, where $\ell$ is a function on the level set satisfying $\ell_u=\frac{|\mathbf{w}|^2}{2u^2}+L+O(\rho^{3-2q})$, where $\mathbf{w}=(w_1,\cdots,w_{n-1})$ and $L$ is a linear function on $\Sigma$.
\end{lemma}
\begin{proof} 
Throughout the proof, we use the coordinate system $\{u,w_1,...,w_{n-1}\}$. We denote $e_\alpha=\frac{\partial}{\partial w_\alpha}$.

It follows directly from the argument in \cite[Lemma 6.5]{HZ24} that $d^\Sigma Y_\alpha = 0$ and $Y_\alpha = \nabla^\Sigma \ell$. More preciously, 
\cite[Equation 6.11]{HZ24} implies $Y_{\alpha,\beta}-Y_{\beta,\alpha}$ is constant on the level set.
Thus, applying Proposition \ref{Y decay}, we get $Y_{\alpha,\beta}-Y_{\beta,\alpha}=\mathcal{O}(\rho^{1-n-c})$. Therefore, $Y_\alpha=\nabla^\Sigma_\alpha \ell$ for some function $\ell$ satisfying $\ell=-\frac{|\mathbf{w}|^2}{u}+\mathcal{O}_2(\rho^{3-n-c})$.

Next, 
we will estimate $\partial_u \ell$.  
Recall that
\begin{equation} \label{kab}
   \nabla^\Sigma_{\alpha\beta}\ell=\frac{1}{2}(Y_{\alpha,\beta}+Y_{\beta,\alpha})= |\nabla u|^{-1}k_{\alpha\beta}
\end{equation}
Note that $k=g+O_1(r^{-q})$ and 
\begin{align*}
    \nabla_\nu (k_{\alpha\beta})=&(\nabla_\nu k)_{\alpha\beta}+k(\nabla_\nu e_\alpha,e_\beta)+k( e_\alpha,\nabla_\nu e_\beta)
    \\=& \langle \nabla_\nu e_\alpha, e_\beta\rangle+\langle e_\alpha,\nabla_\nu e_\beta \rangle+O(r^{-q})
    \\=&O(r^{-q}).
\end{align*}
Recall that the second fundamental form of $\Sigma$ is $-k|_\Sigma$ , we have  $\nabla_\gamma e_\beta= k_{\gamma\beta} \nu$, then
\begin{equation*}
    \nabla_\gamma (k_{\alpha\beta})=(\nabla_\gamma k)_{\alpha\beta}
    +k(\nabla_\gamma e_\alpha, e_\beta)+k( e_\alpha, \nabla_\gamma e_\beta)=k_{\nu\beta}k_{\gamma\alpha}+k_{\nu\alpha}k_{\gamma\beta}+O(r^{-q})=O(r^{-q}).
\end{equation*}
Therefore, $\nabla^\Sigma_{\alpha\beta}\ell=|v|^{-1}\delta_{\alpha\beta}+O_1(|v|^{-1}r^{-q})$.

Since $\partial_u=|\nabla u|^{-1}\nu+Y_\alpha e_\alpha$, 
$|\partial_u|=(|Y|^2+|\nabla u|^{-2})^\frac{1}{2}=(|v|^{-2}+\rho^2+O(\rho^{-2q+2}))^\frac{1}{2}=\mathcal{O}(\rho)$, $|\nabla u|=|u|+\mathcal{O}(\rho^{1-2q})$ and $|Y|=\mathcal{O}(\rho)$, we have
\begin{equation*}
    \begin{split}
        \nabla^\Sigma_{\alpha\beta} \partial_u \ell=& \partial_u (k_{\alpha\beta}|\nabla u|^{-1})
    \\=& \partial_u (|\nabla u|^{-1}\delta_{\alpha\beta}+O_1(r^{-q}|v|^{-1}))
    \\=&  |\nabla u|^{-1} \nu (|\nabla u|^{-1}\delta_{\alpha\beta})+Y_\gamma e_\gamma(|\nabla u|^{-1})\delta_{\alpha\beta}+\mathcal{O}(\rho^{1-2q})
    \\=& k_{\nu\nu}|\nabla u|^{-2}\delta_{\alpha\beta}+
   Y_\gamma k_{\gamma \nu}|\nabla u|^{-1}\delta_{\alpha\beta}
    +\mathcal{O}(\rho^{1-2q})
    \\=&|u|^{-2}\delta_{\alpha\beta}+\mathcal{O}(\rho^{1-2q}).
    \end{split}
\end{equation*}
Integrating the above equation twice, we obtain
\begin{equation*}
    \partial_u \ell =\frac{1}{2}|u|^{-2}|\mathbf{w}|^2+L+\mathcal{O}(\rho^{3-2q}).
\end{equation*}

\end{proof} 
Since the decay rates of the metric coefficients in the coordinate system $\{w_1,\cdots,w_{n-1},u\}$, have been established, we can now employ the maximum principle, combined with the DEC condition $\mu\ge 0$, to prove the main theorem.
\begin{proof}[Proof of \cref{thm:rigidity}]
Combining \cref{Y decay} and \cref{lemma:ell function}, we can embed $(M^n,g,k)$ into $(M^n\times \R,\bar{g})$, where 
\begin{align*}
    \bar{g}=&2d\tau du+g
    \\=& 2d\tau du+(|\nabla u|^{-2}+|Y|^2)du^2+2\sum_{\alpha=1}^{n-1}\frac{\partial \ell}{\partial w_\alpha} du dw_\alpha+ \sum^{n-1}_{\alpha=1} w_\alpha^2
    \\=&2d(\tau +\ell)du+(|\nabla u|^{-2}+|Y|^2-2\partial_u\ell)du^2+\sum_{\alpha=1}^{n-1}dw_\alpha^2.
\end{align*}

Finally, we utilize \cite[Lemma 6.6]{HZ24} which states 
$\mu=-\frac{1}{2}|\nabla u|^{2}\Delta_\Sigma F$, where $F=|\nabla u|^{-2}+|Y|^2-2\ell_u$.

Since $Y_\alpha=-w_\alpha u^{-1}+\mathcal{O}(\rho^{2-2q})$ and $|\nabla u|^{-2}=|u|^{-2}+\mathcal{O}(\rho^{-2q+1})$, we have 
\begin{equation*}
    F+2L=|u|^{-2}+\mathcal{O}(\rho^{3-2q}).
\end{equation*}
Recall that $q>\frac{n}{2}$ and $\Delta_\Sigma F\le 0$, 
it follows by the maximum principle that $F+2L=|u|^{-2}$, implying that $\bar{g}$ is flat.
Therefore, taking $\mathcal{N}$ to infinity so that the slabs $\{-\mathcal{N}\le v\le -\frac{1}{\mathcal{N}}\}$ exhaust $M^n$, we deduce that $(M^n,g,k)$ is a slice in  Minkowski spacetime.
\end{proof}

\begin{proof}[Proof of Corollary \ref{cor:no pp-wave}]
Each asymptotically hyperboloidal slice $(M,g,k)$ contained in a pp-wave is spin and has vanishing mass.
See \cite[Corollary 3.7]{HZ24} for the proof in the asymptotically flat case which extends verbatim to our setting.
Hence, we may apply Theorem \ref{thm:rigidity} to conclude that $(M,g,k)$ isometrically embeds into Minkowski space, i.e. the only vacuum pp-wave spacetime.
\end{proof}


\section{Hyperbolic slices in pp-waves}\label{S:example}

In this section, we discuss an example of a hyperbolic slice within a pp-wave spacetime.

Let $\eta(u)$ be a cutoff function supported on $[-2,-1]$, and let $\rho$ be the radial function $\rho $ on $\mathbb R^{n-1}$.
Consider the spacetime $(\overline M^{n+1},\overline g)$ given by $\overline M^{n+1}=\mathbb R^{n+1}$ and $\overline g=2dudt+Fdu^2+g_{\mathbb R^{n-1}}$ where $F=1+\eta(u)f(\rho)$, and where $f$ is a superharmonic functions with $f(\rho)=\rho^{3-n}$ outside a compact set.
Such spacetimes have first been studied by L.-H.~Huang and D.~Lee in \cite{HL2024}.
Note that $(\overline M^{n+1},\overline g)$ equals Minkowski space everywhere outside the support of $\eta$.
More precisely, $F$ equals $1$ outside the support of $\eta$, and rewriting $x_n=u+t$, we obtain that $\overline g=-dt^2+dx_n^2+g_{\mathbb R^{n-1}}$.
The spacetime $(\overline M^{n+1},\overline g)$ is depicted in Figure \ref{Figure1} where the support of $\eta$ is highlighted in orange.
Note that this spacetime models a wave moving at the speed of light.

\begin{figure}[H]
    \centering

\begin{tikzpicture}\label{fig1}
    \begin{axis}[
        axis lines = middle,
        xlabel = \(x_n\),
        ylabel = {\(t\)},
        domain=-3:3,
        samples=100,
        ymin=0, ymax=5,
        xtick=\empty,
        ytick=\empty,
        enlargelimits=true
    ]

    \addplot [
        thick,
        blue,
        domain=-3:3
    ] {sqrt(x^2 + 1)};

    \addplot [
        fill=orange,
        fill opacity=0.3,
        draw=none
    ]
    coordinates {(-3,-2) (-3,-1) (3,5) (3,4)} -- cycle;

    \addplot [
        thick,
        red,
        domain=-3:3
    ] {x + 1};

    \addplot [
        thick,
        red,
        domain=-3:3
    ] {x + 2};

    \end{axis}
\end{tikzpicture}

    \caption{Schematic depiction of a pp-wave spacetime. The wave is traveling at the speed of light along the region shaded in orange.
    Moreover, the spacetime contains an initial data set $(M^n,g,k)$, highlighted in blue, which fails to be asymptotically hyperboloidal}.
    \label{Figure1}
\end{figure}
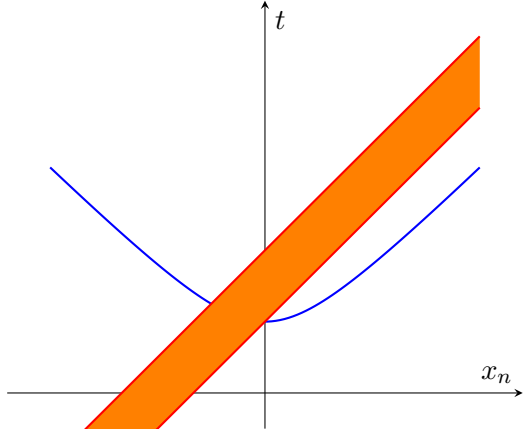

To further explain the above spacetime, we make the following thought experiment:
take an observer centered at the origin. 
Then for $t=0$, the observer is in vacuum till after some time a wave starts to pass through which stretches everything in $x_n$ direction.
Once this wave has passed, the observer will be in vacuum again.

Within that spacetime consider above the slice at $t=0$ the graph given by $\phi(x)=\sqrt{1+|x|^2}$.
In Figure \ref{Figure1} this is initial data set $(M^n,g,k)$ given by this graph is highlighted in blue.
Note that the intersection of $(M^n,g,k)$ with the orange region where $F$ does not equal $1$ is given by a family of horospheres.
From our construction we know that $(F-1)$ decays like $\rho^{3-n}$ on each of these horospheres.
By the same argument as in Section \ref{SS: asymptotic analysis}, we see that $(F-1)$ only decays like $r^{\frac{3-n}2}$ in the initial data set $(M^n,g,k)$.
In particular $(M^n,g,k)$ is not asymptotically hyperboloidal.
This is a stark contrast to the asymptotically flat setting. There is no additional decay rate double process, and therefore there do exist asymptotically flat initial data sets within $(\overline M^{n+1},\overline g)$.

It is an interesting question whether the above constructed initial data set $(M^n,g,k)$ admits perturbations which locally increase the term $\mu-|J|$.
Note that since $(M^n,g,k)$ is not asymptotically hyperboloidal, the argument presented in this manuscript does not apply.

Finally, we now that by replacing $\phi(x)$ with $\psi(x)=\sqrt{1+|x|^2}+2$, we do obtain an asymptotically hyperboloidal initial data set within the pp-wave spacetime.
Note that this initial data set is the hyperbolic space.
In particular, hyperbolic space isometrically embeds into non-trivial pp-wave spacetimes.


\appendix
\section{Proof of the PMT inequality for asymptotically hyperboloidal spin initial data sets}\label{Appendix:PMT}
    Let $(M^n,g,k)$ be a $C^{2,a}_{-q}$-asymptotically hyperboloidal initial data set. We use the convention that $k_{ij}=g(\widetilde{\nabla}_{e_i}e_0,e_j)$ and recall the decay condition of $k$ is given as
    \begin{align*}
        \eta:=k-g=O_1(r^{-q})
    \end{align*}
    Recall that the spinor connection on the the spacetime spinor bundle $\overline{\mathcal{S}}:=\mathcal{S}\oplus\mathcal{S}$ over $(M,g,k)$ is defined by
    \begin{align*}
        \widetilde{\nabla}_j \psi=\nabla_j\psi+\frac{1}{2}k_{jl}e_le_0\psi,
    \end{align*}
    and the corresponding Dirac operators are given by
    \begin{align*}
        \widetilde{\slashed{D}}\psi=\slashed{D}\psi-\frac{1}{2}(\text{tr}_g k) e_0\psi.
    \end{align*}
    Here, we include the integral Lichnerowicz formula, which is the key ingredient for the spinor proof of the PMT.
    \begin{lemma}\label{lemma:Lichnerowicz formula}
        Let $\Omega$ be a bounded open set with smooth boundary in $(M,g,k)$. Then for any $\psi\in C^\infty(\overline{\mathcal{S}})$, we have
        \begin{equation*}
            \begin{split}
                \int_\Omega \left(|\widetilde{\nabla}\psi|^2-|\widetilde{\slashed{D}}\psi|^2+\frac{1}{2}\langle \psi, (\mu+Je_0)\psi\rangle\right)=\int_{\partial\Omega}\langle \psi, \widetilde{\nabla}_\nu \psi+\nu\widetilde{\slashed{D}}\psi\rangle.
            \end{split}
        \end{equation*}
    \end{lemma}
    \begin{proof}
        Choose an orthonormal basis $e_1,\ldots,e_n$ that is parallel at a fixed point. By direct computation, we have
        \begin{align*}
            &\widetilde{\nabla}^*_i=-\nabla_i+\frac{1}{2}k_{ij}e_je_0,\\
            &\widetilde{\slashed{D}}^*=\slashed{D}\psi-\frac{1}{2}(\text{tr}_g k) e_0\psi.
        \end{align*}
        Then, it follows that
        \begin{align*}
            \widetilde{\nabla}^*\widetilde{\nabla}&=\left(-\nabla_i+\frac{1}{2}k_{ij}e_je_0\right)\left(\nabla_i+\frac{1}{2}k_{ij}e_je_0\right)\\
            &=\nabla^*\nabla-\nabla_i(\frac{1}{2}k_{ij}e_je_0)+\frac{1}{2}k_{ij}e_je_0\nabla_i+\frac{1}{4}k_{ij}e_je_0k_{il}e_le_0\\
            &=\nabla^*\nabla-\frac{1}{2}(\nabla_i k_{ij})e_je_0+\frac{1}{4}|k|^2_g\\
            &=\nabla^*\nabla-\frac{1}{2}\left(-\frac{1}{2}|k|^2_g+(\text{div }k)e_0\right).
        \end{align*}
        and we have
        \begin{align*}
            \widetilde{\slashed{D}}^*\widetilde{\slashed{D}}&=\slashed{D}^2-\frac{1}{2}e_i\nabla_i((\text{tr}_g k) e_0)-\frac{1}{2}(\text{tr}_g k) e_0e_i\nabla_i+\frac{1}{4}(\text{tr}_g k)^2\\
            &=\left(\nabla^*\nabla+\frac{1}{4}R\right)-\frac{1}{2}\nabla_i(\text{tr}_g k)e_ie_0+\frac{1}{4}(\text{tr}_g k)^2\\
            &=\nabla^*\nabla+\frac{1}{2}\left(\frac{1}{2}(R+(\text{tr}_g k)^2)-\nabla(\text{tr}_g k)e_0\right).
        \end{align*}
        Combining these two equations, we have
        \begin{align*}
            \widetilde{\slashed{D}}^*\widetilde{\slashed{D}}-\widetilde{\nabla}^*\widetilde{\nabla}&=\frac{1}{2}\left(\frac{1}{2}(R+(\text{tr}_g k)^2-|k|^2_g)+(\text{div }k-\nabla (\text{tr}_g k))e_0\right)\\
            &=\frac{1}{2}\left(\mu+Je_0\right),
        \end{align*}
        which concludes the proof.
    \end{proof}

    \subsection{Preparation}\label{SS:isomorphism} Denote by $M_{ext}$ the designated AH end of $M$. Denote $e=g-b$ and $\eta=k-g$ where $b$ is the standard hyperbolic metric. Let $\mathcal{A}$ be the symmetric endomorphism defined over the tangent bundle $TM_{ext}$ by $g(\mathcal{A}\cdot,\mathcal{A}\cdot)=b(\cdot,\cdot)$. Then we have
    \begin{align*}
        \mathcal{A}=Id-\frac{1}{2}e+O(e^2).
    \end{align*}
    For any $g$-orthonormal frame $\{e_i\}_{i=1}^n$, we denote by $\{\epsilon_i\}_{i=1}^n$ the $b$-orthonormal frame obtained by $\epsilon_i=\mathcal{A}^{-1}e_i$.
    
    We can use $\mathcal{A}$ as an isomorphism between the spinor frame bundles of $b$ and $g$ as well as the spacetime spinor frame bundles by defining $\mathcal{A}e_0=e_0$.  Now we can transfer the spin connection $\mathring{\nabla}$ of $b$ to a connection $\overline{\nabla}:=\mathcal{A}\mathring{\nabla}\mathcal{A}^{-1}$ on the spinor bundle of $g$. Then the connection $1$-forms respectively of $\overline{\nabla}$ and $\nabla$ are given by
    \begin{align*}
        \overline{\omega}_{ij}=g(\overline{\nabla} e_i,e_j) \text{ for }i,j=1,\ldots,n,\\
        \omega_{ij}=g(\nabla e_i,e_j) \text{ for }i,j=1,\ldots,n,
    \end{align*}
    Then for any spinor $\varphi=\varphi^s\sigma_s$ on $M_{ext}$, we can write
    \begin{align*}
        (\nabla-\overline{\nabla})\varphi=\frac{1}{4}\sum_{i,j=1}^n(\omega_{ij}-\overline{\omega}_{ij}) e_i e_j\varphi
    \end{align*}

    \subsection{Spacetime Killing spinor and the mass}
    Recall that an spacetime Killing spinor $\sigma$ on $\mathbb{H}^n$ is defined as the spinor satisfying 
    \begin{align*}
        \mathring{\nabla}_X\sigma+\frac{1}{2}X e_0 \sigma=0.
    \end{align*}
    It is straightforward to check that the function $V_\sigma=|\sigma|^2_b$ is a static potential of $\mathbb{H}^n$. Let $\chi$ be a smooth function on $M$ that is equal to $0$ outside $M_{ext}$ and $1$ for large $r$ in $M_{ext}$.
    \begin{proposition}[Witten's mass formula]
    \label{mass formula}
        Let $\psi=\chi\mathcal{A}\sigma+\psi_0$ be a spacetime spinor on $M$ where $\psi_0$ is a compactly supported spinor. Let $V_\sigma=|\sigma|_b^2$. Then we have
        \begin{align*}
            \mathcal{H}_{hyp}(V_\sigma)&=4\lim_{r\to\infty}\int_{S_r} \left\langle \widetilde{\nabla}_{\mathcal{A}\nu_r}\psi+\mathcal{A}\nu_r\tilde{\slashed{D}}\psi,\psi\right\rangle\\
            &=4\int_{M} |\widetilde{\nabla}\psi|^2-|\tilde{\slashed{D}\psi}|^2+\frac{1}{2}\langle \psi, (\mu+Je_0)\psi\rangle
        \end{align*}
        Here, $\nu_r$ is the $b$-normal of the coordinate spheres $S_r$.
    \end{proposition}

    \begin{proof}
        The second equality follows from \cref{lemma:Lichnerowicz formula}. We now compute the boundary integral, which is similar to \cite[Proposition 3.3]{Maerten2006}. For convenience, we choose the $g$-orthonormal frame such that $e_1=\mathcal{A}\nu_r$. We then have
        \begin{align*}
        \widetilde{\nabla}_{\mathcal{A}\nu_r}\psi+\mathcal{A}\nu_r\tilde{\slashed{D}}\psi&=\mathcal{A}\nu_r\left(\sum_{j=2}^n e_j\widetilde{\nabla}_{j}\right)\psi.
        \end{align*}
        Next, we compute
        \begin{align*}
            \widetilde{\nabla}_X (\mathcal{A}\sigma)=(\nabla_X-\overline{\nabla}_X)(\mathcal{A}\sigma)+\left(\overline{\nabla}_X+\frac{1}{2}k(X) e_0\right)(\mathcal{A}\sigma).
        \end{align*}
        Since $\overline{\nabla}_X(\mathcal{A}\sigma)=\mathcal{A}\mathring{\nabla}_X\sigma=\mathcal{A}(-\frac{1}{2}X e_0 \sigma_0)=-\frac{1}{2}\mathcal{A}X e_0\mathcal{A}\sigma_0$, we obtain
        \begin{align*}
            \widetilde{\nabla}_X (\mathcal{A}\sigma)&=\left(\nabla_X-\overline{\nabla}_X\right)(\mathcal{A}\sigma)+\frac{1}{2}\left(k(X) e_0-\mathcal{A}X e_0\right) (\mathcal{A}\sigma).
        \end{align*}
        Thus the boundary term becomes for $r$ large enough
        \begin{align}\label{eqn:boundary_term}
            &\sum_{j=2}^n\left\langle\mathcal{A}\nu_r e_j\left((\nabla_{j}-\overline{\nabla}_j)+\frac{1}{2}\left(k(e_j) e_0- \mathcal{A}e_j e_0\right)\right)(\mathcal{A}\sigma),\mathcal{A}\sigma\right\rangle\nonumber\\
            &=\sum_{j=2}^n\left\langle\mathcal{A}\nu_r e_j\left((\nabla_{j}-\overline{\nabla}_j)+\frac{1}{2}\left(\eta(e_j) e_0+e_0(\mathcal{A}-Id) e_j\right)\right)(\mathcal{A}\sigma),\mathcal{A}\sigma\right\rangle
        \end{align}
        Recall that $\{\epsilon_j=\mathcal{A}^{-1}e_j\}_{j=1}^n$ is a $b$-orthonormal frame. First, we have
        \begin{align*}
            \sum_{j=2}^n\langle \mathcal{A}\nu_r e_j \eta(e_j) e_0(\mathcal{A}\sigma),\mathcal{A}\sigma\rangle &=\sum_{j=2}^n\langle \mathcal{A}\nu_r\mathcal{A}\epsilon_j \eta(\mathcal{A}\epsilon_j) \mathcal{A}e_0(\mathcal{A}\sigma),\mathcal{A}\sigma\rangle\\
            & =\sum_{j=2}^n\langle \nu_r\epsilon_j (\mathcal{A}^{-1}\circ \eta\circ \mathcal{A})(\epsilon_j) e_0\sigma,\sigma\rangle_b
        \end{align*}
        Note that $\mathcal{A}^{-1}\circ \eta\circ \mathcal{A}=\eta+\frac{1}{2}e\circ \eta-\frac{1}{2}k\circ e+O(|e|^2)$. By the decay condition of $\eta$, it follows that $\mathcal{A}^{-1}\circ k\circ \mathcal{A}=k+O(r^{-2q})$ as $r\to\infty$. In addition, we have
        \begin{align*}
            \nu_r\sum_{j=2}^n\epsilon_j \eta(\epsilon_j)&=\epsilon_1\left(\sum_{j=1}^n \epsilon_j \eta(\epsilon_j)-\epsilon_1 \eta(\epsilon_1)\right)\\
            &=\eta(\nu_r)-(\text{tr}_b \eta)\nu_r,
        \end{align*}
        which implies
        \begin{equation}\label{eqn:boundary_term2}
            \begin{split}
                \sum_{j=2}^n\langle \mathcal{A}\nu_r e_j \eta(e_j) e_0(\mathcal{A}\sigma),\mathcal{A}\sigma\rangle&=\langle \big(\eta(\nu_r)-(\text{tr}_b\eta)\nu_r\big) e_0\sigma,\sigma\rangle_b+O(r^{-2q})\\
                &=((\text{tr}_b \eta)dV_\sigma-i_{dV_\sigma}\eta)(\nu_r)+O(r^{-2q})
            \end{split}
        \end{equation}
        where we use
        \[
            dV_\sigma(X)=X\langle \sigma,\sigma\rangle_b=2\langle \mathring{\nabla}_X \sigma,\sigma\rangle_b=-\langle X e_0\sigma,\sigma\rangle_b.
        \]

        Secondly, we compute
        \begin{align}\label{eqn:boundary_term3}
            &\sum_{j=2}^n\langle \mathcal{A}\nu_r e_j e_0\left(\mathcal{A}-Id\right)e_j(\mathcal{A}\sigma),\mathcal{A}\sigma\rangle\nonumber\\
            &\qquad =\sum_{j=2}^n\langle \mathcal{A}\nu_r\mathcal{A}\epsilon_j\mathcal{A}e_0\left(\mathcal{A}-Id\right)e_j(\mathcal{A}\sigma),\mathcal{A}\sigma\rangle\nonumber\\
            &\qquad =\sum_{j=2}^n\langle \nu_r\epsilon_j e_0(\mathcal{A}^{-1}\left(\mathcal{A}-Id\right)\mathcal{A})(\epsilon_j)\sigma,\sigma\rangle_b\\
            &\qquad =-\frac{1}{2}\sum_{j=2}^n\langle \nu_r\epsilon_j e_0 e(\epsilon_j)\sigma,\sigma\rangle_b+O(r^{-2q})\nonumber\\
            &\qquad =\frac{1}{2} \langle (e(\nu_r)-\text{tr}_b e(\nu_r)+O(r^{-2q})) e_0\sigma,\sigma\rangle=\frac{1}{2} ((\text{tr}_b e)dV_\sigma-i_{dV_\sigma}e)(\nu_r)+O(r^{-2q}).\nonumber
        \end{align}

        Lastly, consider the term
        \begin{align*}
            &\sum_{j=2}^n\langle \mathcal{A}\nu_r e_j(\nabla_{e_j}-\overline{\nabla}_{e_j})(\mathcal{A}\sigma),\mathcal{A}\sigma\rangle\\
            &\qquad =\frac{1}{4}\sum_{j=2}^n\sum_{k,l=1}^n\langle (\omega_{kl}-\overline{\omega}_{kl})(e_j)\mathcal{A}\nu_r e_j e_k e_l \mathcal{A}\sigma,\mathcal{A}\sigma\rangle\\
            &\qquad =\frac{1}{4}\left(\sum_{j,k,l=2}^n \langle (\omega_{kl}-\overline{\omega}_{kl})(e_j)\mathcal{A}\nu_r e_j e_k e_l \mathcal{A}\sigma,\mathcal{A}\sigma\rangle\right.\\
            &\qquad\qquad\qquad \left.+2\sum_{j,k=2}^n \langle (\omega_{1k}-\overline{\omega}_{1k})(e_j)\mathcal{A}\nu_r e_j e_1 e_k \mathcal{A}\sigma,\mathcal{A}\sigma\rangle\right)\\
            &\qquad =\frac{1}{4}\left(\sum_{j,k,l=2}^n \langle (\omega_{kl}-\overline{\omega}_{kl})(e_j)\mathcal{A}\nu_r e_j e_k e_l \mathcal{A}\sigma,\mathcal{A}\sigma\rangle\right.\\
            &\qquad\qquad\qquad \left.+2\sum_{j,k=2}^n \langle (\omega_{1k}-\overline{\omega}_{1k})(e_j) e_j e_k \mathcal{A}\sigma,\mathcal{A}\sigma\rangle\right)=:\frac{1}{4}(I_1+2I_2).
        \end{align*}
        Since the resulting integral is real, we only need to consider the real part of the above expression. The first term $I_1$ can be estimated as the following: using the skew-symmetric property of $\omega_{kl}-\overline{\omega}_{kl}$ and the fact that the summand is purely imaginary if $k=l$, we have
        \begin{align*}
            &\text{Re}(I_1)=\sum_{\substack{j,k,l=2\\ j,k,l\text{: distinct}}}^n \langle (\omega_{kl}-\overline{\omega}_{kl})(e_j)\mathcal{A}\nu_r e_j e_k e_l \mathcal{A}\sigma,\mathcal{A}\sigma\rangle
        \end{align*}
        However, one may observe that
        \begin{align}\label{eqn:boundary_term4}
            (\omega_{kl}-\overline{\omega}_{kl})(e_j)\epsilon_j\epsilon_k\epsilon_l&=\frac{1}{2}b(\mathcal{A}^{-1}(\mathring{\nabla}_{e_j}\mathcal{A})\epsilon_k-\mathcal{A}^{-1}(\mathring{\nabla}_{e_k}\mathcal{A})\epsilon_j,\epsilon_l)\epsilon_j\epsilon_k\epsilon_l\nonumber\\
            &=b(\mathcal{A}^{-1}(\mathring{\nabla}_{e_j}\mathcal{A})\epsilon_k,\epsilon_l)\epsilon_j\epsilon_k\epsilon_l\\
            &=-\frac{1}{2}b((\mathring{\nabla}_{\epsilon_j} e)\epsilon_k,\epsilon_l)\epsilon_j\epsilon_k\epsilon_l+O(r^{-2q})=O(r^{-2q}),\nonumber
        \end{align}
        thus $I_1$ has no contribution in the limit. 
        For $I_2$, we have
        \begin{align*}
            I_2&=\sum_{j,k=2}^n \langle (\omega_{1k}-\overline{\omega}_{1k})(e_j) e_j e_k \mathcal{A}\sigma,\mathcal{A}\sigma\rangle\\
            &=-\sum_{j=2}^n \langle (\omega_{1j}-\overline{\omega}_{1j})(e_j)\sigma,\sigma\rangle+\sum_{\substack{j,k=2\\ j\neq k}}^n \langle (\omega_{1k}-\overline{\omega}_{1k})(e_j) e_j e_k \mathcal{A}\sigma,\mathcal{A}\sigma\rangle.
        \end{align*}
        Notice that the second sum is purely imaginary, so we have
        \begin{align}\label{eqn:boundary_term5}
            \text{Re}(I_2)&=-\left(\sum_{j=2}^n (\omega_{1j}-\overline{\omega}_{1j})(e_j)\right)V_\sigma \nonumber\\
            &=V_\sigma\sum_{j=2}^n g((\mathring{\nabla}_{e_1}\mathcal{A})\mathcal{A}^{-1}e_j-(\mathring{\nabla}_{e_j}\mathcal{A})\mathcal{A}^{-1}e_1,e_j)\nonumber\\
            &=-\frac{1}{2}V_\sigma\sum_{j=2}^n b(\mathcal{A}^{-1}(\mathring{\nabla}_{\epsilon_1} e)\epsilon_j-\mathcal{A}^{-1}(\mathring{\nabla}_{\epsilon_j}\mathcal{A})\epsilon_1,\epsilon_j)+O(r^{-2q+1})\\
            &=-\frac{1}{2}V_\sigma\sum_{j=2}^n b((\mathring{\nabla}_{\epsilon_1} e)\epsilon_j-(\mathring{\nabla}_{\epsilon_j} e)\epsilon_1,\epsilon_j)+O(r^{-2q+1})\nonumber\\
            &=\frac{1}{2}V_\sigma(\text{div}_b (e)(\nu_r)-d(\text{tr}_b (e)))(\nu_r)+O(r^{-2q+1}).\nonumber
        \end{align}
        Combining \eqref{eqn:boundary_term}, \eqref{eqn:boundary_term2}, \eqref{eqn:boundary_term3}, \eqref{eqn:boundary_term4}, and \eqref{eqn:boundary_term5}, we obtain that
        \begin{align*}
            &\lim_{r\to\infty}\int_{S_r} \left\langle \widetilde{\nabla}_{\mathcal{A}\nu_r}\psi+\mathcal{A}\nu_r\tilde{\slashed{D}}\psi,\psi\right\rangle\nonumber\\
            &\quad =\frac{1}{4}\lim_{r\to\infty}\int_{S_r} V_\sigma(\text{div}_b(e)-d\text{tr}_b(e))+\text{tr}_b(e+2\eta)dV_\sigma-(e+2\eta)(\nabla^b V,\cdot)=\frac{1}{4}\mathcal{H}_{hyp}(V_\sigma).
        \end{align*}

    \end{proof}

\subsection{Spinor existence}
We follow the argument of \cite{Zhang2004}. Let $C_c^\infty(\overline{\mathcal{S}})$ be the space of smooth spinor fields on $M$ with compact support. Let $H_0^1(\overline{\mathcal{S}})$ be the closure of $C_c^\infty(\overline{\mathcal{S}})$ in $H^1_{\text{loc}}$ with respect to the norm induced by the $H^1$-norm, as
\begin{align*}
    ||\phi||_{H^1}^2=||\phi||_{L^2}^2+||\nabla\phi||_{L^2}^2,
\end{align*}
where $\nabla$ is the Levi-Civita connection of $g$.

Define a bounded sesquilinear form $\mathcal{B}$ on $C_c^\infty(\overline{\mathcal{S}})$ by
\begin{align*}
    \mathcal{B}(\phi,\psi)=\int_{M} \langle \tilde{\slashed{D}}\phi,\tilde{\slashed{D}}\psi\rangle.
\end{align*}
Then by the Lichnerowicz formula, we have
\begin{equation*}
\begin{split}
    \mathcal{B}(\phi,\phi)&=\int_M |\widetilde{\nabla}\phi|^2+\frac{1}{2}\langle \phi,(\mu+Je_0)\phi\rangle\\
    &\ge \int_M |\nabla\phi|^2+\sum_i \text{Re}\langle \nabla_i \phi,k_{ij}e_je_0\phi\rangle+\frac{1}{4}|k_{ij}e_je_0\phi|^2\\
    &\ge \int_M \frac{1}{2}|\nabla\phi|^2+\sum_{i}\frac{1}{2}|\nabla_i \phi+k_{ij}e_je_0\phi|^2-\frac{1}{4}|k_{ij}e_je_0\phi|^2\\
    &\ge \frac{1}{2}||\phi||_{H^1}^2-C||\phi||_{L^2}^2,
\end{split}
\end{equation*}
where $C$ only depends on $k$ and the dimension $n$. Note that the second line uses the dominant energy condition. By the above inequality, we can continuously extend $\mathcal{B}$ to $H_0^1(\overline{\mathcal{S}})$ as a (not necessarily strict) coercive sesquilinear form. 

Let $\sigma$ be an spacetime Killing spinor, which is a solution to the equation
\begin{align*}
    \mathring{\nabla}_X \sigma+\frac{1}{2}X e_0\sigma=0.
\end{align*}
Recall $\chi$ is a cut-off function that vanishes outside of $M_{ext}$ and equals $1$ for $r$ large enough. Due to the decay conditions, we have
\begin{align*}
    \chi\mathcal{A}\sigma\notin L^2, \quad \widetilde{\nabla}(\chi\mathcal{A}\sigma)\in L^2, \quad \tilde{\slashed{D}}(\chi\mathcal{A}\sigma)\in L^2.
\end{align*}
Since $\eta=k-g=O(r^{-q})$, we have the following property:

\begin{lemma}\label{lemma:trivial kernel}
    If $\phi\in L^2$ and $\widetilde{\nabla}\phi\equiv 0$, then $\phi\equiv 0$.
\end{lemma}
\begin{proof}
    This can be seen as follows: $\widetilde{\nabla}\phi\equiv 0$ implies that
    \[ \big|\partial_i |\phi|^2\big|\le (1+|\eta|)|\phi|^2.\]
    Suppose that $\phi\not\equiv 0$. Choose $x_0$ outside a large ball so that $|\phi(x_0)|\ne 0$ and $|\eta|\le C<1$ for some constant $C$. We then have 
    \[
        |\phi(x)|^2\ge |\phi(x_0)|^2\exp{\big(-(1+C)d_g(x,x_0)\big)}.
    \] Hence, $\phi$ cannot be in $L^2$ for $n\ge 3$, therefore, we conclude that $\phi\equiv 0$.
\end{proof}

Now we prove the following existence result:
\begin{proposition}\label{prop:existence of spinor}
    There exists a unique spinor field $\psi_0\in H_0^1(\overline{\mathcal{S}})$ such that
    \begin{align*}
        \tilde{\slashed{D}}(\chi\mathcal{A}\sigma+\psi_0)=0.
    \end{align*}
\end{proposition}
\begin{proof}
    Since $\mathcal{B}$ is coercive in $H_0^1(\overline{\mathcal{S}})$ and $\tilde{\slashed{D}}(\chi\mathcal{A}\sigma)\in L^2$, we can apply the Lax-Milgram theorem to obtain a weak solution $\psi_0\in H_0^1(\overline{\mathcal{S}})$ to the equation 
    \[
        \tilde{\slashed{D}}^*\tilde{\slashed{D}}\psi_0=-\tilde{\slashed{D}}^*\tilde{\slashed{D}}(\chi\mathcal{A}\sigma).
    \]
    Note that such a weak solution is unique because the kernel of $\tilde{\slashed{D}}$ in $H_0^1(M,\overline{\mathcal{S}})$ is trivial by the previous observation.
    
    Define $\psi:=\chi\mathcal{A}\sigma+\psi_0$ and $\phi:=\tilde{\slashed{D}}\psi$. The elliptic regularity theory implies that $\phi\in H^1_0(\overline{\mathcal{S}})$ and $\tilde{\slashed{D}}^*\phi=0$ in the classical sense. By the Lichnerowicz formula, we have $\widetilde{\nabla}\phi=0$ as well. Then by \cref{lemma:trivial kernel}, we conclude $\phi\equiv 0$. Therefore, $\psi_0$ is the desired solution.
\end{proof}

Combining Propositions \ref{mass formula} and \ref{prop:existence of spinor}, we conclude the following:
\begin{theorem}\cite{CJL2004,Zhang2004}
    Let $(M,g,k)$ be an $C^{2,a}_{-q}$ asymptotically hyperboloidal initial data set satisfying (DEC). Then $E\ge|\vec{P}|$.
\end{theorem}
\begin{proof}
    Choose the spacetime Killing spinor $\sigma$ on $\mathbb{H}^n$ such that we have $\mathcal{H}_{hyp}(V_\sigma)=E^2-|\vec{P}|^2$ for $V_\sigma:=|\sigma|^2_b$. By \cref{prop:existence of spinor}, there exists $\psi=\chi\mathcal{A}\sigma+\psi_0$ with $\psi_0\in H^1_0(\overline{\mathcal{S}})$ such that $\tilde{\slashed{D}}\psi=0$. Let $V=|\psi|^2$. Combining \cref{lemma:Lichnerowicz formula} and \cref{mass formula}, we have
    \begin{align*}
        \mathcal{H}_{hyp}(V_\sigma)&=\mathcal{H}_{hyp}(V)\\
        &=4\lim_{r\to\infty}\int_{S_r} \left\langle \widetilde{\nabla}_{\mathcal{A}\nu_r}\psi+\mathcal{A}\nu_r\tilde{\slashed{D}}\psi,\psi\right\rangle\\
        &=4\int_M \left(|\widetilde{\nabla}\psi|^2+\frac{1}{2}\langle \psi, (\mu+Je_0)\psi\rangle\right)\ge 0
    \end{align*}
\end{proof}

\section{Data availability and conflict of interest statement}

The authors declare that there is no conflict of interest and that no data availability issues exist related to this research.

\bibliographystyle{alpha}
\bibliography{literature.bib}

\newcommand{\etalchar}[1]{$^{#1}$}
\begin{thebibliography}{BHK{\etalchar{+}}23}

\bibitem[ACG08]{ACG2008}
Lars Andersson, Mingliang Cai, and Gregory~J. Galloway.
\newblock Rigidity and positivity of mass for asymptotically hyperbolic manifolds.
\newblock {\em Ann. Henri Poincar\'{e}}, 9(1):1--33, 2008.

\bibitem[Bau89]{Baum1989}
Helga Baum.
\newblock Complete {R}iemannian manifolds with imaginary {K}illing spinors.
\newblock {\em Ann. Global Anal. Geom.}, 7(3):205--226, 1989.

\bibitem[BC96]{BeigChrusciel}
Robert Beig and Piotr~T. Chru\'sciel.
\newblock Killing vectors in asymptotically flat space-times. {I}. {A}symptotically translational {K}illing vectors and the rigid positive energy theorem.
\newblock {\em J. Math. Phys.}, 37(4):1939--1961, 1996.

\bibitem[BC17]{BieriChrusciel}
Lydia Bieri and Piotr~T. Chru\'sciel.
\newblock Future-complete null hypersurfaces, interior gluings, and the {T}rautman-{B}ondi mass.
\newblock In {\em Nonlinear analysis in geometry and applied mathematics}, volume~1 of {\em Harv. Univ. Cent. Math. Sci. Appl. Ser. Math.}, pages 1--31. Int. Press, Somerville, MA, 2017.

\bibitem[BHK{\etalchar{+}}23]{BHKKZ}
Hubert Bray, Sven Hirsch, Demetre Kazaras, Marcus Khuri, and Yiyue Zhang.
\newblock Spacetime harmonic functions and applications to mass.
\newblock In {\em Perspectives in scalar curvature. Vol. 2}, pages 593--639. Singapore: World Scientific, 2023.

\bibitem[CD19]{CD2019}
Piotr~T. Chru{\'s}ciel and Erwann Delay.
\newblock {The hyperbolic positive energy theorem}.
\newblock {\em arXiv}, 2019.

\bibitem[CG21]{CG2021}
Piotr~T. Chru{\'s}ciel and Gregory~J. Galloway.
\newblock {Positive mass theorems for asymptotically hyperbolic Riemannian manfiolds with boundary}.
\newblock {\em arXiv}, 2021.

\bibitem[CGNP18]{CGNP2018}
Piotr~T. Chru\'{s}ciel, Gregory~J. Galloway, Luc Nguyen, and Tim-Torben Paetz.
\newblock On the mass aspect function and positive energy theorems for asymptotically hyperbolic manifolds.
\newblock {\em Classical Quantum Gravity}, 35(11):115015, 38, 2018.

\bibitem[CH03]{CH2003}
Piotr~T. Chru\'{s}ciel and Marc Herzlich.
\newblock The mass of asymptotically hyperbolic {R}iemannian manifolds.
\newblock {\em Pacific J. Math.}, 212(2):231--264, 2003.

\bibitem[CJLe04]{CJL2004}
Piotr~T. Chru\'sciel, Jacek Jezierski, and Szymon \L\c~eski.
\newblock The {T}rautman-{B}ondi mass of hyperboloidal initial data sets.
\newblock {\em Adv. Theor. Math. Phys.}, 8(1):83--139, 2004.

\bibitem[CM06]{CM2006}
Piotr~T. Chru\'sciel and Daniel Maerten.
\newblock Killing vectors in asymptotically flat space-times. {II}. {A}symptotically translational {K}illing vectors and the rigid positive energy theorem in higher dimensions.
\newblock {\em J. Math. Phys.}, 47(2):022502, 10, 2006.

\bibitem[CMT06]{CMT2006}
Piotr~T. Chru\'sciel, Daniel Maerten, and Paul Tod.
\newblock Rigid upper bounds for the angular momentum and centre of mass on non-singular asymptotically anti-de {S}itter space-times.
\newblock {\em J. High Energy Phys.}, (11):084, 42, 2006.

\bibitem[CWY16]{CWY2016}
Po-Ning Chen, Mu-Tao Wang, and Shing-Tung Yau.
\newblock Conserved quantities on asymptotically hyperbolic initial data sets.
\newblock {\em Adv. Theor. Math. Phys.}, 20(6):1337--1375, 2016.

\bibitem[Dir28]{Dirac}
P.~A.~M. Dirac.
\newblock The quantum theory of the electron. {I}.
\newblock {\em Proceedings of the Royal Society of London. Series A}, 117:610--624, 1928.

\bibitem[God91]{godbillon1991feuilletages}
Claude Godbillon.
\newblock {\em Feuilletages: {\'e}tudes g{\'e}om{\'e}triques}, volume~98 of {\em Prog. Math.}
\newblock Basel etc.: Birkh{\"a}user Verlag, 1991.

\bibitem[HJ22]{HJ2022}
Lan-Hsuan Huang and Hyun~Chul Jang.
\newblock Scalar curvature deformation and mass rigidity for {ALH} manifolds with boundary.
\newblock {\em Trans. Amer. Math. Soc.}, 375(11):8151--8191, 2022.

\bibitem[HJM20]{HJM2020}
Lan-Hsuan Huang, Hyun~Chul Jang, and Daniel Martin.
\newblock Mass rigidity for hyperbolic manifolds.
\newblock {\em Comm. Math. Phys.}, 376(3):2329--2349, 2020.

\bibitem[HKK22]{HKK}
Sven Hirsch, Demetre Kazaras, and Marcus Khuri.
\newblock Spacetime harmonic functions and the mass of 3-dimensional asymptotically flat initial data for the {Einstein} equations.
\newblock {\em J. Differ. Geom.}, 122(2):223--258, 2022.

\bibitem[HKKZ25]{HKKZ}
Sven Hirsch, Demetre Kazaras, Marcus Khuri, and Yiyue Zhang.
\newblock Rigid comparison geometry for riemannian bands and open incomplete manifolds.
\newblock {\em Mathematische Annalen}, 391(2):2587--2652, 2025.

\bibitem[HL20]{HL2020}
Lan-Hsuan Huang and Dan~A. Lee.
\newblock Equality in the spacetime positive mass theorem.
\newblock {\em Comm. Math. Phys.}, 376(3):2379--2407, 2020.

\bibitem[HL24]{HL2024}
Lan-Hsuan Huang and Dan~A. Lee.
\newblock Bartnik mass minimizing initial data sets and improvability of the dominant energy scalar.
\newblock {\em J. Differential Geom.}, 126(2):741--800, 2024.

\bibitem[HZ23]{HZ23}
Sven Hirsch and Yiyue Zhang.
\newblock The case of equality for the spacetime positive mass theorem.
\newblock {\em J. Geom. Anal.}, 33(1):13, 2023.
\newblock Id/No 30.

\bibitem[HZ24]{HZ24}
Sven Hirsch and Yiyue Zhang.
\newblock Initial data sets with vanishing mass are contained in pp-wave spacetimes.
\newblock Preprint, {arXiv}:2403.15984 [gr-qc] (2024), 2024.

\bibitem[Lee13]{JohnLee}
John~M. Lee.
\newblock {\em Introduction to smooth manifolds}, volume 218 of {\em Grad. Texts Math.}
\newblock New York, NY: Springer, 2nd revised ed edition, 2013.

\bibitem[Lun23]{Lundberg2023}
David Lundberg.
\newblock On {J}ang's equation and the positive mass theorem for asymptotically hyperbolic initial data sets with dimensions above three and below eight, 2023.

\bibitem[Mae06]{Maerten2006}
Daniel Maerten.
\newblock Positive energy-momentum theorem for {A}d{S}-asymptotically hyperbolic manifolds.
\newblock {\em Ann. Henri Poincar\'{e}}, 7(5):975--1011, 2006.

\bibitem[Opp30]{Oppenheimer}
J~Robert Oppenheimer.
\newblock On the theory of electrons and protons.
\newblock {\em Physical Review}, 35(5):562, 1930.

\bibitem[Pen82]{Penrose1982}
Roger Penrose.
\newblock Some unsolved problems in classical general relativity.
\newblock In {\em Seminar on {D}ifferential {G}eometry}, volume No. 102 of {\em Ann. of Math. Stud.}, pages 631--668. Princeton Univ. Press, Princeton, NJ, 1982.

\bibitem[Sak21]{Sakovich2021}
Anna Sakovich.
\newblock The {J}ang {E}quation and the {P}ositive {M}ass {T}heorem in the {A}symptotically {H}yperbolic {S}etting.
\newblock {\em Comm. Math. Phys.}, 386(2):903--973, 2021.

\bibitem[SY81]{SY1981}
Richard Schoen and Shing~Tung Yau.
\newblock Proof of the positive mass theorem. {II}.
\newblock {\em Comm. Math. Phys.}, 79(2):231--260, 1981.

\bibitem[SY82]{SchoenYauBondi}
Richard Schoen and Shing~Tung Yau.
\newblock Proof that the {B}ondi mass is positive.
\newblock {\em Phys. Rev. Lett.}, 48(6):369--371, 1982.

\bibitem[Wan01]{Wang2001}
Xiaodong Wang.
\newblock The mass of asymptotically hyperbolic manifolds.
\newblock {\em J. Differential Geom.}, 57(2):273--299, 2001.

\bibitem[Wit81]{Witten1981}
Edward Witten.
\newblock {A new proof of the positive energy theorem}.
\newblock {\em Communications in Mathematical Physics}, 80(3):381--402, 1981.

\bibitem[WX15]{WX2015}
Yaohua Wang and Xu~Xu.
\newblock Hyperbolic positive energy theorem with electromagnetic fields.
\newblock {\em Classical Quantum Gravity}, 32(2):025007, 20, 2015.

\bibitem[XZ08]{XZ2008}
Naqing Xie and Xiao Zhang.
\newblock Positive mass theorems for asymptotically {A}d{S} spacetimes with arbitrary cosmological constant.
\newblock {\em Internat. J. Math.}, 19(3):285--302, 2008.

\bibitem[Zha04]{Zhang2004}
Xiao Zhang.
\newblock A definition of total energy-momenta and the positive mass theorem on asymptotically hyperbolic 3-manifolds. {I}.
\newblock {\em Comm. Math. Phys.}, 249(3):529--548, 2004.

\end{thebibliography}

\end{document}